\documentclass[10pt,a4paper]{article}
\usepackage[top=3cm, bottom=3cm, left=2cm, right=2cm]{geometry}
\usepackage{ART-SFVEM,styleBiot}
\newtheorem{theorem}{Theorem}

\newtheorem{remark}{Remark}

  \newtheorem{problem}{Problem}
\usepackage[
  backend=biber,    
  doi=true,         
  url=false,        
  isbn=false        
]{biblatex}
\author{Francesca Marcon, Stefano Scialò}
\title{Hybrid-Dimensional Biot Problem with an Optimization Based Domain Decomposition Approach}
\date{}
\begin{document}
\maketitle
\begin{abstract}
The present work proposes a numerical approach for solving coupled flow and mechanics problems in fractured porous media, represented as mixed-dimensional domains. In this formulation, the elements of the 3D mesh are allowed to arbitrarily intersect the fractures. Displacements are discontinuous across fractures through the use of the eXtended Finite Element Method (XFEM) on the 3D mesh. The mechanical problem is formulated as a saddle-point problem, in which Lagrange multipliers are used to enforce displacement continuity across the fractures. The resulting Lagrange multipliers represent the stress field acting on the fracture surfaces. Likewise, the pressure field is allowed to be discontinuous across fractures through the XFEM formulation on the non-conforming mesh and is computed using an optimization-based domain decomposition strategy specifically designed for mixed-dimensional problems. The fixed-stress splitting scheme is employed to decouple the flow and mechanics subproblems, while the mixed-dimensional pressure problem is solved at each fixed-stress iteration using the Conjugate Gradient (CG) method. The combination of the fixed-stress scheme and the CG solver proves to be highly effective for this class of problems. 
\end{abstract}
\section{Introduction}

In recent years, coupling flow and mechanical deformation has gained increasing attention due to its significant role in many important applications, such as geothermal energy exploitation, long-term storage of carbon dioxide (CO$_2$), and the study of subsurface hydrology and seismic activity in the Earth's crust.
The rock masses considered are often characterized by the presence of fracture networks, which have a significant impact on both the hydraulic and mechanical behavior of the subsurface. The inclusion of such features in a numerical model is therefore fundamental for the accurate representation of the hydro-mechanical properties of the medium.
Discrete fracture-matrix (DFM) models represent a suitable choice for the explicit representation of fractures in a porous medium. DFMs are derived from continuum equi-dimensional models by averaging the governing equations over the fracture domain in the normal direction and, hence, reducing the representation of fractures to codimension-one manifolds embedded in the porous matrix. This reduces the computational cost of the simulations by avoiding the meshing of the thin fracture regions, but additional equations are introduced at the fracture-matrix interface to close the problem. Furthermore, even after the geometrical reduction, DFM models may still exhibit complex geometries due to the presence of intricate fracture networks.
A classical approach for modelling the coupling between flow and geomechanics is the Biot model, which yields a fully coupled system of equations. The direct solution of the resulting discrete system, the so-called monolithic strategy, can be computationally prohibitive for large-scale simulations. A popular alternative consists of applying sequential operator-splitting techniques, in which the solution is obtained iteratively by solving, in sequence, properly defined subproblems. The most widely used method in this context is the fixed-stress split \cite{KIM2011CMAME}, which exhibits good convergence properties and robustness \cite{Mikelic2013,Storvik2019}. The fixed-stress split scheme introduces a stabilization term into the flow equation. Alternatively, a stabilization term can be added to the mechanics equation, yielding the so-called undrained splitting scheme \cite{KIM2011Undrained}. Other splitting schemes can also be found in \cite{TURSKA1993}.
A wide range of numerical schemes for porous-media problems with fractures can be found in the literature. In recent years, a number of significant works have focused on developing different modelling approaches for poro-mechanical coupling. For example, in \cite{Garipov2016}, the authors investigate a mixed finite element–finite volume discretization of a 3D–2D nonlinear Biot model; in \cite{BYZ2017}, a 2D–1D Brinkman-Biot formulation is analyzed, deriving it from a full 2D–2D model through a topological reduction approach; in \cite{Berge2020}, the authors study a finite volume discretization of the 3D Biot equations coupled with contact mechanics conditions; and in \cite{Brezina2024}, a finite element discretization of the Biot equations is considered both in the 3D domain and in the 2D fracture, together with a sensitivity analysis of the two fixed-stress stability parameters.
To the best of our knowledge, the majority of the approaches discussed above require the generation of a mesh conforming to the fractures.
The present work builds on the poro-mechanical model proposed in \cite{JHA2014}. In that work, a multiphase flow model is coupled with nonlinear mechanics for the displacement field. Fracture slip is computed iteratively by balancing the stresses on the fracture plane. Displacements and pressures are discontinuous, and the in-fracture pressure is computed as the maximum of the pressures on the two sides of the fracture. The mesh is conforming to the fractures, and the fixed-stress scheme is used for decoupling.

The objective of this work is to develop a numerical approach for the Biot problem in which the flow subproblem is described as a mixed dimensional problem, in which fractures are represented as planar interfaces \cite{MJR2005} and the in-fracture pressure is used to compute the effective stress on the fracture planes \cite{JHA2014}. 
The fixed-stress operator-splitting scheme enables the flow and mechanical problems to be effectively decoupled. After this decoupling, however, the flow problem remains a coupled mixed-dimensional problem, since the unknowns are the pressures in both the 3D bulk domain and the 2D fracture network. This problem is solved through a domain decomposition method based on numerical optimization, suitable for the use of non-conforming meshes and parallel computing. The combined use of an operator-splitting for decoupling flow and mechanics, and of a domain decomposition method on non conforming meshes for the mixed dimensional flow problem is the main novelty of this work, which is expected to significantly reduce the computational cost of simulations for coupled poro-mechanics in mixed-dimensional domains. 

To present the method, we adopt the model proposed in \cite{JHA2014}, with the simplifying assumption of single-phase flow. The discrete problem is derived using non-conforming meshes for both the flow and mechanics subproblems, through XFEM \cite{XfemReview} to represent discontinuous functions on non-conforming meshes.
The mixed-dimensional flow problem is solved using an optimization-based approach with a five-field formulation suitable for domain decomposition and parallel computing \cite{S2024Finel}. The optimization problem is solved numerically via a conjugate gradient scheme. Unlike the flow problem, the mechanics subproblem is not formulated in mixed-dimensional form, as no in-fracture problem is solved for the displacement field. It is written, following \cite{JHA2014,Aagaard2013}, as a saddle-point problem and is solved monolithically, in order to keep the focus of the present work on the mixed-dimensional flow model.
The proposed approach is supported by a convergence analysis of the fixed-stress iteration and by numerical experiments assessing its viability and effectiveness.

The paper is organized as follows. Section~\ref{sec:modelpb} describes the model problem of interest. Section~\ref{sec:discrform} presents its discrete formulation together with a proof of the convergence of the fixed-stress method for this class of problems. Sections~\ref{sec:FlowProblem} and \ref{sec:mechProblem} describe the solution strategies for the flow and mechanics subproblems. Numerical results are presented in Section~\ref{sec:numres}, and conclusions are drawn in Section~\ref{sec:Conclusions}.

\section{Model problem}\label{sec:modelpb}

Let us consider a 3D domain $\D$, representing a block of porous material crossed by a set of non-intersecting planar interfaces $F^j$, with $j=1,\ldots,\numFrac$. Let $\Dint$ denote the domain $\D$ excluding the fractures, i.e. $\Dint := \D\setminus \{\cup_j F^j\}$.
Let $\Gamma^i_D$ and $\Gamma^i_N$, with $i=m$ for the mechanics problem and $i=f$ for the flow problem, denote the portions of the boundary $\partial \D$ where Dirichlet and Neumann conditions apply, respectively. We have that, for $i=m$ or $f$, $\bar\Gamma^i_D \cup \bar\Gamma^i_N=\partial \D$ and $\Gamma^i_D\cap\Gamma^i_N = \emptyset$.
Moreover, for any $j=1,\ldots,\numFrac$ let $\gamma^j_D$ and $\gamma^j_N$ denote the portions of the boundary $\partial F^j$ where Dirichlet and Neumann conditions apply, hence $\bar\gamma^j_D \cup \bar\gamma^j_N=\partial F^j$ and $\gamma^j_D\cap\gamma^j_N = \emptyset$.

The boundary of $\Dint$ is instead split into three parts, $\partial \Dint := \mathring{\Gamma}^i_D \cup  \mathring{\Gamma}^i_N \cup \{\cup_j\omega^j\}$, with $i=m$ for the mechanics problem and $i=f$ for the flow problem and defined such that $\mathring{\Gamma}^i_D = \partial \Dint \cap \Gamma^i_D$, $\mathring{\Gamma}^i_N = \partial \Dint \cap \Gamma^i_N$ and, for each $j=1,\ldots, \numFrac$, $\omega^j =  \partial \Dint \cap F^j$, this last part being the portion of the boundary of $\Dint$ that coincides with the $j$-th fracture $F^j$. 

For each $j=1,\ldots,\numFrac$, denoting by $\bs{n}^j$ the unit normal to the $j$-th fracture $F^j$, with fixed chosen orientation, we will further distinguish the two ``sides'' of any $\omega^j$ as $\omega^j_+$ and $\omega^j_-$ , with $\omega^j_+$ being the portion of $\omega^j$ with outward pointing unit normal equal to $\bs{n}^j$, i.e. $\bs{n}^j_+:=\bs{n}^j_{\omega^j_+}=\bs{n}^j$ and $\bs{n}^j_-:=\bs{n}^j_{\omega^j_-}=-\bs{n}^j$.

The linear Biot model of interest, thus,
results as follows.
\begin{problem}\label{pb:model}
Find the displacement $\bs{u}$, the 3D pressure $p^{\D}$, the 2D pressure $p^{F^j}$ and the effective normal traction $\bs{\Lambda}^j$(a Lagrange multiplier see \cite{JHA2014}) on each fracture  with $j=1,\ldots,\numFrac$ such that
\begin{align}
\frac{\partial}{\partial t} \left(\frac{1}{M}\, p^{\D}+\alpha\, \div \bs{u}\right) - \div\left(\KK{\D} \nabla p^{\D}\right) = g \quad \mbox{ in } \Dint \times (0, T]\,,
\\
- \div\left(\KK{F^j} \nabla p^{F^j}\right)  +\KK{\D}\nabla p^{\D,j}_{+}\cdot \bs{n}^j_+ +\KK{\D}\nabla p^{\D,j}_-\cdot \bs{n}^j_- =0 
\quad \mbox{ in } F^j \times (0, T]\;, \forall j=1,\ldots,\numFrac\,,
\\
\KK{\D}\nabla p^{\D,j}_{\pm}\cdot \bs{n}^j_{\pm} = -\eta^j \left(p^{\D,j}_{\pm} - p^{F^j} \right) \quad \mbox{ in } \omega^j_{\pm} \times (0, T]\;, \forall j=1,\ldots,\numFrac\,,
\\
\alpha \nabla p^{\D} - \div \bs{\sigma}(\bs{u}) = 0  \quad \mbox{ in } \Dint \times (0, T]\,,
\\
\left(\bs{\sigma}^j_{\pm}(\bs{u})- \alpha p^{\D,j}_{\pm}\bs{I}\right) \cdot  \bs{n}^j_{\pm} = \bs{\Lambda}^j - \alpha p^{F^j}  \bs{n}^j_{\pm}\quad \mbox{ in } \omega^j_{\pm} \times (0, T]\;, \forall j=1,\ldots,\numFrac\,,
\\
\bs{u}^j_+ = \bs{u}^j_-
\quad \mbox{ in } F^j \times (0, T]\;, \forall j=1,\ldots,\numFrac\,,
\end{align}
where $1/M$ is the compressibility constant, $\alpha$ is the Biot-Willis constant, $\KK{\D}=k^\D/\mu_f$, with $k^\D$ the  hydraulic permeability in $\Dint$ and $\mu_f$ fluid viscosity, while, for each fracture $F^j$, with $j=1,\ldots,\numFrac$, $\KK{F^j}={k^j_\parallel}/{\mu_f}$ and $\eta^j={k_{\perp}^j}/{\mu_f}$ being $k^j_\parallel$ and $k^j_\perp$, respectively, the effective hydraulic permeability of the fracture in the tangential and normal directions . The quantity $g$ represents a forced fluid extraction or injection process and $\bs{I}$ is the second order identity tensor. 
Since $p^{\D}$, $\bs{\sigma}(\bs{u})$ and $\bs{u}$ are considered sufficiently regular, for each $j=1,\ldots,\numFrac$ $p^{\D,j}_+$ and $p^{\D,j}_-$, $\bs{\sigma}^j_+$ and $\bs{\sigma}^j_-$, $\bs{u}^j_+$ and $\bs{u}^j_-$ denote the trace of $p^{\D}$, $\bs{\sigma}$ and $\bs{u}$ on $\omega^j_+$ and on $\omega^j_-$ respectively.

Moreover, the stress tensor $\bs{\sigma}$ and the strain tensor $\bs{\epsilon}$ satisfy the constitutive relations of linear elasticity, i.e.
\begin{equation}\label{eq:ElasticityConstitutiveRel}
\begin{aligned}
\bs{\sigma} (\bs{u}) = \bs{C} \bs{\epsilon}(\bs{u}) = \lambda tr(\bs{\epsilon}(\bs{u})) \bs{I} + 2\mu\bs{\epsilon}(\bs{u})\,,
\\
\bs{\epsilon}(\bs{u}) = \frac{1}{2}\left(\nabla \bs{u} +\nabla \bs{u} ^\intercal\right)\,,
\end{aligned}
\end{equation}
where $\lambda$ and $\mu$ are the Lamé coefficients.

For the sake of simplicity, let us consider homogeneous initial and boundary conditions
\begin{align}
    p^\D &= 0 &\mbox{ on } \mathring{\Gamma}^f_D \times (0, T]\,,
    \\
    \KK{\D} \nabla p^{\D} \cdot \bs{n}_{\Gamma^f_N} &= 0 &\mbox{ on } \mathring{\Gamma}^f_N \times (0, T]\,,
    \\
    p^{F^j} &= 0 &\mbox{ on } \gamma^j_D \times (0, T]\;, \forall j=1,\ldots,\numFrac\,,
    \\
   \KK{F^j} \nabla p^{F^j} \cdot \bs{n}_{\gamma^j_N} &= 0 &\mbox{ on } \gamma^j_N \times (0, T]\;, \forall j=1,\ldots,\numFrac\,,
   \\
   \bs{u} &= 0 &\mbox{ on } \mathring{\Gamma}^m_D \times (0, T]\,,
   \\
   \bs{\sigma}(\bs{u}) \cdot \bs{n}_{\Gamma^m_N} &=0 &\mbox{ on } \mathring{\Gamma}^m_N \times (0, T]\,,
   \\
  p^\D &= 0 &\mbox{ in } \mathring{\D} \mbox{ and } t=0\,,
  \\
p^{F^j} &= 0 &\mbox{ in } F^j  \mbox{ and } t=0\;, \forall j=1,\ldots,\numFrac\,,
  \\
  \bs{u}&= 0 &\mbox{ in } \mathring{\D} \mbox{ and } t=0\,.
\end{align}

We remark that the extension to nonhomogeneous initial and boundary conditions is well-established in the literature and follows the classical approach of Galerkin methods.
\end{problem}
We make the following assumptions on the physical coefficients of the model problem:
\begin{enumerate}[label=\textbf{A.\arabic*}]
    \item \label{assumptionCoeff_omega} Let $\alpha,\lambda,\mu$ and $\KK{\D}$ be strictly positive constants, and ${1}/{M}$ be nonnegative.
    \item \label{assumptionCoeff_frac} For each fracture $F^j$ with $j=1,\ldots,\numFrac$, let $\eta^j$ and $\KK{F^j}$ be strictly positive constants.
\end{enumerate}

Let us now move to the weak formulation of the previous problem.
Let  $V^{f,\D}=\{q^\D\in \sobh{1}{\Dint}: q^\D = 0 \mbox{ on }\Gamma^f_D\}$ and for each $j=1,\ldots,\numFrac$ let $ V^{f,{F^j}}=\{q^{F^j}\in \sobh{1}{F^j}: q^{F^j} = 0 \mbox{ on }\gamma^j_D\}$ be the spaces of admissible 3D pressure and 2D pressure on each fracture, respectively, 
and $\bs{V}^m=\{\bs{v}\in [\sobh{1}{\Dint}]^3: \bs{v} = 0 \mbox{ on } \Gamma^m_D\}$ be the space of admissible displacements.

To simplify the analysis, we introduce for each $j=1,\ldots,\numFrac$ the total fracture traction:
\begin{equation}
  \hat{\bs{\Lambda}}^j=\bs{\Lambda}^j-\alpha p^{F^j}\bs{n}^j,
\end{equation}
Hence, for any fracture $F^j$, with $j=1,\ldots,\numFrac$, naturally let $\bs{Q}^{m,j} = \left([\sobh{\frac12}{F^j}]\right)'$ be the space of admissible total fracture tractions. 

The variational formulation of Problem \ref{pb:model} results as follows.
\begin{problem}\label{pb:contVar}
Find $\bs{u}\in\lebl{0,T;\bs{V}^m}$, $p^\D\in\lebl{0,T;V^{f,\D}}$ and for each $j=1,\ldots,\numFrac$, $p^{F^j}\in V^{f,{F^j}}$ and $\bs{\Lambda}^j\in \bs{Q}^{m,j}$ for a.e. $t\in(0,T)$ such that:
\begin{align}
\begin{aligned}
\int_{\Dint} \frac{1}{M} \frac{\partial p^\D}{\partial t}  q^\D  + \int_{\Dint} \alpha \div \left(\frac{\partial \bs{u}}{\partial t}\right)  q^\D + \int_{\Dint} \KK{\D} \nabla p^\D \, \nabla q^{\D} + \sum_{j=1}^{\numFrac} \int_{\omega^j} \eta^j (p^{\D,j}_+\,q^{\D,j}_+ + p^{\D,j}_-\,q^{\D,j}_-)  \\ -\sum_{j=1}^{\numFrac}  \int_{\omega^j} \eta^j p^{F^j} \left(q^{\D,j}_+ + q^{\D,j}_-\right)
= \int_{\Dint} g\, q^\D  \quad \forall \; q^\D\in V^{f,\D}, \mbox{ for a.e. } t\in (0,T)
\end{aligned}
\\
\begin{aligned}
\int_{\omega^j} \KK{F^j} \nabla p^{F^j} \, \nabla q^{F^j} + \int_{\omega^j} 2 \eta^j\, p^{F^j} \, q^{F^j} - \int_{\omega^j}\eta^j \left(p^{\D,j}_+ + p^{\D,j}_-\right) q^{F^j} = 0  \quad \forall \; q^{F^j}\in V^{f,{F^j}}, \mbox{ for a.e. } t\in (0,T)\\ \mbox{and 
} \forall j=1,\ldots,\numFrac
\end{aligned}
\\
-\int_{\Dint} \alpha p^\D \div \bs{v} + \int_{\Dint} \bs{C} \bs{\epsilon}(\bs{u}) : \bs{\epsilon}(\bs{v}) 
+ \sum_{j=1}^{\numFrac} \dual[\omega^j]{\hat{\bs{\Lambda}}^j}{\bs{v}^j_- - \bs{v}^j_+}
= 0 \quad \forall \; \bs{v}\in \bs{V}^{m}, \mbox{ for a.e. } t\in (0,T)
\\
\dual[\omega^j]{\bs{u}^j_- - \bs{u}^j_+}{\bs{\tau}^j} =0 \quad \forall \; \bs{\tau}^j \in \bs{Q}^{m,j} , \mbox{ for a.e. } t\in (0,T)\; \mbox{and 
} \forall j=1,\ldots,\numFrac \label{pb:contVar:EqLambda}
\end{align}
\end{problem}
\section{Discrete formulation}\label{sec:discrform}
In this section, we derive the discrete counterpart of Problem \ref{pb:contVar}, within the finite element method framework, and using the eXtended Finite Element Method (XFEM) \cite{XfemReview} to represent discontinuous solutions on non conforming meshes, and the backward Euler method for time advancing. 
We introduce and construct, independently, a tetrahedral mesh of $\D$, denoted by $\mathcal{T}^{\D}$ and for each fracture $F^j$, with $j=1,\ldots,\numFrac$, two triangular meshes denoted by $\mathcal{T}^{F^j}$ and $\mathcal{T}^{\Lambda^j}$. 
Let $h_{\D}$, $h_{F^j}$ and $h_{\Lambda^j}$ be the maximum size of the elements in the meshes $\mathcal{T}^{\D}$, $\mathcal{T}^{F^j}$ and $\mathcal{T}^{\Lambda^j}$, respectively.
Notice that $\mathcal{T}^{\D}$ does not conform to the interfaces $\omega^j$. 

We thus exploit the eXtended Finite Element Method (XFEM) \cite{Xfem99_1,xfem99_2} as spatial discretization scheme for the 3D variables.
First, let us focus on the construction of the discretization space for $p^{\D}$, denoted by $V^{f,\D}_h$. 
The discontinuity of $p^{\D}$ across each interface is represented using linear Lagrangian finite elements on $\mathcal{T}^{\D}$ with additional enrichment basis functions describing the discontinuity.
Let $\Poly[f]{1}{\mathcal{T}^{\D}}$ denote the standard global linear Lagrangian finite element space defined on $\mathcal{T}^{\D}$, satisfying the zero Dirichlet condition on $\mathring{\Gamma}^f_D$.
To construct the enrichment basis functions, we follow a similar approach to the one proposed in \cite{SukumarXFEM2000,S2024Finel}. Accordingly, for each fracture, an enrichment basis function is constructed, based on its intersection with the domain.
If the fracture $F^j$ entirely crosses the domain, the enrichment is constructed starting from the Heaviside function $\mathcal{H}^j:\mathbb{R}^3\to\mathbb{R}$ defined for any $\bs{x}\in\mathbb{R}^3$ as $\mathcal{H}^j(\bs{x})=\text{sign}\left(\bs{n}_{F^j}\cdot(\bs{x}-\bs{x}_0)\right)$, where $\bs{x}_0$ is a chosen fixed point on $F^j$.
Instead, if the fracture $F^j$ terminates inside the domain, then the solution $p^{\D}$ is expected to be discontinuous across the fracture but continuous elsewhere.
In this case the enrichment is constructed to exhibit this behaviour.
Let us consider the polygonal shape of the fracture $F^j$, in particular taking into account each edge that lies in the interior of the domain $\D$. Let $n_{\mathring{e}}^j$ denote the number of these internal edges and let, for any $i=1,\ldots,n_{\mathring{e}}^j$, $\varepsilon^j_i$ be the equation of the line containing the $i-$th edge defined in the fracture-local reference system. 
Then we can define the function $\mathcal{E}^j:\mathbb{R}^3\to\mathbb{R}$ as
\begin{equation}
    \mathcal{E}^j(\bs{x})= \begin{cases}
        \zeta \prod_{i=1}^{n_{\mathring{e}}^j} \limits \left(\varepsilon^j_i(\mathcal{B}_{F^j}(\bs{x}))\right)^\varrho &\mbox{if } \mathcal{B}_{F^j}(\bs{x})\in F^j\,,
        \\
        0 &\mbox{otherwise}\,,
    \end{cases}
\end{equation}
where $\zeta$ is a scaling parameter and $\mathcal{B}_{F^j}$ denotes the affine mapping from the reference system in $\D$ to the fracture $F^j$ reference system. 
The scaling parameter $\zeta$ is chosen so that $\mathcal{E}^j$ is equal to one in the fracture barycenter, while the exponent $\varrho$ controls the shape of the enrichment function. 
In the numerical results proposed, it is equal to $\frac12$.
We remark that if the fracture $F^j$ entirely crosses the domain $\mathcal{E}^j$ reduces to the unit constant function, whereas if $F^j$ is entirely contained in $\D$ then $\mathcal{E}^j$ is a bubble function on $F^j$.
Finally, the function $\mathcal{H}^j\mathcal{E}^j$ is chosen as enrichment function when $F^j$ terminates inside the domain.

The discrete space $V^{f,\D}_h$ is defined by adding the enrichment basis functions to the standard finite element basis. 
Let $\mathcal{T}^\D_{\mathcal{H}^j}$ denote the subset of the elements in $\mathcal{T}^\D$ that are entirely cut by the fracture $F^j$ and let $\mathcal{T}^\D_{\mathcal{E}^j}$ be the subset of elements in $\mathcal{T}^\D$ that lie in $\D \setminus\partial \D$ and intersect $\partial F^j$.
Then, let $\bs{x_k}$ denote the $k$-th vertex of $\mathcal{T}^{\D}$, and let $\phi_k$ denote the basis function of $\Poly{1}{\mathcal{T}^{\D}}$, Lagrangian in the $k$-th vertex $\bs{x}_k$, i.e. $\phi_{\ell}(\bs{x}_k)=\delta_{lk}$ where $\delta_{lk}$ denotes the Kronecker delta. We also denote $\mathcal{I}_\mathcal{E}^j$ the set of indexes of mesh vertices belonging to elements in $\mathcal{T}^\D_{\mathcal{E}^j}$ and $\mathcal{I}_\mathcal{H}^j$ the set of indexes of mesh vertices belonging to elements in $\mathcal{T}^\D_{\mathcal{H}^j}$.
We define for any fracture $F^j$, for any $k\in \mathcal{I}_\mathcal{E}^j$:
\begin{equation}
    \Xi^{j}_k (\bs{x}) =\phi_k(\bs{x})\left(\mathcal{H}^j(\bs{x})\mathcal{E}^j(\bs{x})-\mathcal{H}^j(\bs{x_k})\mathcal{E}^j(\bs{x_k})\right), \mbox{ with }\bs{x}\in\mathbb{R}^3
\end{equation}
and, for $k\in \mathcal{I}_\mathcal{H}^j$
\begin{equation}
    \Upsilon^{j}_k  =\phi_k(\bs{x})\left(\mathcal{H}^j(\bs{x})-\mathcal{H}^j(\bs{x_k}))\right), \mbox{ with }\bs{x}\in\mathbb{R}^3
\end{equation}
Finally we get:
\begin{multline}
    V^{f,\D}_h:= \Poly{1}{\mathcal{T}^{\D}} \cup \mbox{span}\left\{\Xi^{j}_k\,, \forall k \in \mathcal{I}_{\mathcal{E}}^j\,, \forall j=1,\ldots,\numFrac \right\}
    \cup \mbox{span}\left\{\Upsilon^{j}_k\,, \forall k \in \mathcal{I}_{\mathcal{H}}^j\,, \forall j=1,\ldots,\numFrac\right\}
\end{multline}

Let $\bs{V}^m_h$ be the discretization space for the displacement solution $\bs{u}$ defined on the 3D tessellation $\mathcal{T}^{\D}$. 
This space is a vectorial space, constructed following, for each component, the XFEM approach described above for the 3D flow discrete space and satisfying the zero Dirichlet conditions on $\mathring{\Gamma}^m_D$.

Moreover, for each fracture $F^j$, we consider the standard linear FEM space defined on $\mathcal{T}^{F^j}$. 
This space is used for the discretization of the 2D pressure on each fracture, enforcing the zero Dirichlet condition on $\gamma^j_D$ and, thus, defining the space $V_h^{f,F^j}$.
Finally, for each fracture $F^j$, we consider the piecewise constant space defined on $\mathcal{T}^{\Lambda^j}$. 
This space is used for the discretization of each component of the total fracture traction, hence defining the space $\bs{Q}_h^{m,j}$. 
Different choices are also possible.
In order to state and analyze the semi-discrete in space problem, we need the following assumption.
 \begin{enumerate}[start=3,label=\textbf{A.\arabic*}]
    \item \label{ass:inf-sup} There exists a constant $\beta$ independent of $h_{\D}$ and $h_{\Lambda}$, such that: for each set of $ \bs{q}^j_h \in \bs{Q}_h^{m,j}, $ with $j=1,\ldots,\numFrac$ 
    \begin{equation}
        \beta \sum_{j=1}^{\numFrac}\norm[\bs{Q}_h^{m,j}]{\bs{q}_h^j}
        \leq 
        \sup_{\bs{v}_h \in \bs{V}^m_h\colon \norm[\bs{V}^m_h]{\bs{v}_h}\neq 0} \frac{\sum_{j=1}^{\numFrac}\int_{\omega^j}\left((\bs{v}_h^j)_- - (\bs{v}_h^j)_+\right)\,\bs{q}_h^j}{\left(2\mu\norm[0,\D]{\bs{\epsilon}(\bs{v}_h)}^2+\lambda\norm[0,\D]{\div\bs{v}_h}^2\right)^{\frac{1}{2}}}.
    \end{equation} 
\end{enumerate}
\begin{remark}
    Under the hypothesis of non intersecting and sufficiently distant fractures, Assumption \ref{ass:inf-sup} follows from this condition: 
    there exists a constant $\beta$ independent of $h_{\D}$ and $h_{\Lambda}$, such that for each $j=1,\ldots,\numFrac$ and $\bs{q}_h^j\in \bs{Q}_h^{m,j}$
    \begin{equation}
        \beta \norm[\bs{Q}_h^{m,j}]{\bs{q}_h^j}
        \leq 
        \sup_{\bs{v}_h \in \bs{V}^m_h\colon \norm[\bs{V}^m_h]{\bs{v}_h}\neq 0} \frac{\int_{\omega^j}\left((\bs{v}_h^j)_- - (\bs{v}_h^j)_+\right)\,\bs{q}_h^j}{\left(2\mu\norm[0,\D]{\bs{\epsilon}(\bs{v}_h)}^2+\lambda\norm[0,\D]{\div\bs{v}_h}^2\right)^{\frac{1}{2}}}\,.
    \end{equation} 
For isolated interfaces this condition can be satisfied by choosing a sufficiently coarse mesh for the Lagrange multipliers or through the addition of stabilizing terms \cite{Burman2014}.
In the numerical tests we will choose a sufficiently fine mesh for $\mathcal{T}^{\D}$ to consider each fracture as isolated, and a sufficiently coarse mesh for $\mathcal{T}^{\Lambda^j}$ such that the discrete mechanics subproblem results well posed. To the best of our knowledge, a proof of the general case for arbitrarily close or intersecting interfaces is still not available and far beyond the scope of the present work. 
\end{remark}

The semi-discrete in space variational formulation of Problem \ref{pb:contVar} results as follows.
\begin{problem}\label{pb:discreteVar}
Find $\bs{u}_h\in\lebl{0,T;\bs{V}_h^m}$, $p_h^\D\in\lebl{0,T;V_h^{f,\D}}$ and for each $j=1,\ldots,\numFrac$ $p_h^{F^j}\in\lebl{0,T;V_h^{f,{F^j}}}$ and $\hat{\bs{\Lambda}}_h^j\in\bs{Q}_h^{m,j}$ for a.e. $t\in (0,T)$ such that 
\begin{align}
\begin{aligned}
\int_{\Dint} \frac{1}{M} \frac{\partial p_h^\D}{\partial t}  q_h^\D  + \int_{\Dint} \alpha \div \left(\frac{\partial \bs{u}_h}{\partial t}\right)  q_h^\D + \int_{\Dint} K^{\D} \nabla p_h^\D \, \nabla q_h^{\D} + \sum_{j=1}^{\numFrac} \int_{\omega^j} \eta^j \left((p_h^{\D,j})_+\,(q_h^{\D,j})_+ + (p^{\D,j}_h)_-\,(q^{\D,j}_h)_-\right)
\\ -\sum_{j=1}^{\numFrac}\int_{\omega^j} \eta^j p_h^{F^j} \left((q_h^{\D,j})_+ + (q_h^{\D,j})_-\right)
= \int_{\Dint} g\, q_h^\D  \quad \forall \; q_h^\D\in V_h^{f,\D}, \mbox{ for a.e. } t\in (0,T)
\end{aligned}
\\
\begin{aligned}
\int_{\omega^j} K^{F^j} \nabla p_h^{F^j} \, \nabla q_h^{F^j} + \int_{\omega^j} 2 \eta^j\, p_h^{F^j} \, q_h^{F^j} - \int_{\omega^j} \eta^j \left((p_h^{\D,j})_+ + (p^{\D,j}_h)_-\right) q_h^{F^j} = 0  \quad \forall \; q_h^{F^j}\in V_h^{f,F^j}, 
\\ \mbox{ for a.e. } t\in (0,T) \mbox{ and 
} \forall j=1,\ldots,\numFrac
\end{aligned}
\\
\begin{aligned}
-\int_{\Dint} \alpha p_h^\D \div \bs{v}_h + \int_{\Dint} \bs{C} \bs{\epsilon}(\bs{u}_h) : \bs{\epsilon}(\bs{v}_h) 
+ \sum_{j=1}^{\numFrac} \int_{\omega^j}\hat{\bs{\Lambda}}_h^j\,\left((\bs{v}_h^j)_- - (\bs{v}_h^j)_+\right)
= 0 \quad \forall \; \bs{v}_h\in \bs{V}_h^{m}, 
\\
\mbox{ for a.e. } t\in (0,T)
\end{aligned}
\\
\int_{\omega^j}\left((\bs{u}_h^j)_- - (\bs{u}_h^j)_+\right)\,\bs{\tau}_h^j =0 \quad \forall \; \bs{\tau}_h^j \in \bs{Q}_h^{m,j} , \mbox{ for a.e. } t\in (0,T)\; \mbox{and 
} \forall j=1,\ldots,\numFrac
\end{align}
\end{problem}
Applying the backward Euler scheme to the semi-discrete formulation yields the following fully discrete problem.
\begin{problem}\label{pb:backEulerDiscr}
Let $k\geq 1$, $\Delta t>0$ and assume $\bs{u}_h^{k-1}\in\bs{V}^m_h$, $(p_h^\D)^{k-1}\in V_h^{f,\D} $ 
be given.
Find $\bs{u}_h^{k}\in\bs{V}^m_h$, $(p_h^\D)^{k}\in V_h^{f,\D} $ and for each $j=1,\ldots,\numFrac$ $(p_h^{F^j})^{k}\in V_h^{f,F^j}$ and $(\hat{\bs{\Lambda}}_h^j)^{k}\in\bs{Q}_h^{m,j}$ such that
\begin{align}
\begin{aligned} \label{eq:fullDiscr_flow3D}
\frac{1}{\Delta t}\left(\int_{\Dint} \frac{1}{M} (p_h^\D)^{k}  q_h^\D  + \int_{\Dint} \alpha \div \bs{u}_h^{k}  q_h^\D \right)
+ \int_{\Dint} K^{\D} \nabla (p_h^\D)^k \, \nabla q_h^{\D} + \sum_{j=1}^{\numFrac}\int_{\omega^j} \eta^j \left((p_h^{\D,j})^k_+\,(q_h^{\D,j})_+ + (p^{\D,j}_h)^k_-\,(q^{\D,j}_h)_-\right)
\\ -  \sum_{j=1}^{\numFrac}\int_{\omega^j} \eta^j (p_h^{F^j})^k \left((q_h^{\D,j})_+ + (q_h^{\D,j})_-\right)
= \frac{1}{\Delta t}\left(\int_{\Dint} \frac{1}{M} (p_h^\D)^{k-1}  q_h^\D + \int_{\Dint} \alpha \div \bs{u}_h^{k-1}  q_h^\D\right)+ \int_{\Dint} g\, q_h^\D  \quad \forall \; q_h^\D\in V_h^{f,\D}
\end{aligned}
\\
\begin{aligned}\label{eq:fullDiscr_flow2D}
    \int_{\omega^j} K^{F^j} \nabla (p_h^{F^j})^k \, \nabla q_h^{F^j} + \int_{\omega^j} 2 \eta^j\, (p_h^{F^j})^k \, q_h^{F^j} - \int_{\omega^j} \eta^j \left((p_h^{\D,j})^k_+ + (p^{\D,j}_h)^k_-\right) q_h^{F^j} = 0  \quad \forall \; q_h^{F^j}\in V_h^{f,F^j}
\\ \mbox{and 
} \forall j=1,\ldots,\numFrac
\end{aligned}
\\ \label{eq:fullDiscr_mech}
\begin{aligned}
-\int_{\Dint} \alpha (p_h^\D)^k \div \bs{v}_h + \int_{\Dint} \bs{C} \bs{\epsilon}(\bs{u}^k_h) : \bs{\epsilon}(\bs{v}_h) 
+ \sum_{j=1}^{\numFrac} \int_{\omega^j}(\hat{\bs{\Lambda}}_h^j)^k\,\left((\bs{v}_h^j)_- - (\bs{v}_h^j)_+\right)
= 0 \quad \forall \; \bs{v}_h\in \bs{V}_h^{m}, 
\end{aligned}
\\
\int_{\omega^j}\left((\bs{u}_h^j)^k_- - (\bs{u}_h^j)^k_+\right)\,\bs{\tau}_h^j =0 \quad \forall \; \bs{\tau}_h^j \in \bs{Q}_h^{m,j} ,  \mbox{and 
} \forall j=1,\ldots,\numFrac \label{discreteWeakCont}
\end{align}
\end{problem}

\section{Fixed-stress splitting}\label{sec:fixed-stress}

In this section, we describe the application of the fixed-stress scheme \cite{KIM2011CMAME,SETTARI1998_FixedStress} to Problem~\ref{pb:backEulerDiscr}, and prove its convergence.

The decoupling not only reduces the computational cost by allowing the flow and mechanics subproblems to be solved sequentially, but also enables the use of optimized methods and solvers for each subproblem.

The fixed-stress scheme is an iterative splitting method in which the flow subproblem \eqref{eq:fullDiscr_flow3D} and \eqref{eq:fullDiscr_flow2D} is solved first using the displacement from the previous iteration, followed by the mechanics subproblem \eqref{eq:fullDiscr_mech} with the updated pressure. The procedure is repeated until convergence.
To ensure convergence, we need to stabilize the flow equation \eqref{eq:fullDiscr_flow3D} by adding the term $L((p_h^\D)^{k,i}-(p_h^\D)^{k,i-1})$, where $i$ denotes the iteration index and $L$ is a stability parameter. 
We follow the approach of choosing $L$ with a physical motivation (\cite{KIM2011CMAME}), hence $L=\frac{\alpha^2}{K_{dr}}$, where $K_{dr}$ represents the physical drained bulk modulus, i.e. $K_{dr} = \lambda + \frac{2}{3}\mu$.
This stabilization results in fixing the volumetric stress, i.e., imposing that
\begin{equation}
    K_{dr}\div\bs{u}_h^{k,i} =  K_{dr}\div\bs{u}_h^{k,i-1} +\alpha\left((p^\D_h)^{k,i}-(p^\D_h)^{k,i-1}\right)\,.
\end{equation}
We recall that $K_{dr}$ satisfies the inequality
\begin{equation}\label{eq:relationDrainedBulk}
    2\mu\norm[0,\D]{\bs{\epsilon}(\bs{v}_h)}^2+\lambda\norm[0,\D]{\div\bs{v}_h}^2 \geq K_{dr} \norm[0,\D]{\div\bs{v}_h}^2 \quad \forall \bs{v}_h\in\bs{V}^m_h\,.
\end{equation}

Now, let us apply this scheme to Problem \ref{pb:backEulerDiscr}. Let $k\geq 1$, and assume that $\bs{u}_h^{k-1}\in\bs{V}^m_h$, $(p_h^\D)^{k-1}\in V_h^{f,\D} $ and $(p_h^{F^j})^{k-1}\in V_h^{f,F^j}$ for any $j =1,\ldots,\numFrac$, are given. 
We choose the solution at the previous time step as the initial guess, i.e., for any $k$ $\bs{u}_h^{k,0}=\bs{u}_h^{k-1}$, $(p_h^\D)^{k,0}=(p_h^\D)^{k-1}$ and $(p_h^{F^j})^{k,0}=(p_h^{F^j})^{k-1}$ for any $j =1,\ldots,\numFrac$.
\begin{problem}\label{pb:FixedStress-dec}
    Let $i\geq 1$ be the iteration index and assume $\bs{u}_h^{k-1}$, $(p_h^\D)^{k-1} $, $(p_h^{F^j})^{k-1}$, and $\bs{u}_h^{k,i-1}$, $(p_h^\D)^{k,i-1} $, $(p_h^{F^j})^{k,i-1}$ $\forall\, j=1,\ldots,\numFrac$ are given.
    Find  $(p_h^\D)^{k,i}\in V_h^{f,\D} $ and for each $j=1,\ldots,\numFrac$ $(p_h^{F^j})^{k,i}\in V_h^{f,{F^j}}$ such that
    \begin{align}\label{eq:flow3d_fs}
        \begin{aligned}
\frac{1}{\Delta t}\int_{\Dint} \left(\frac{1}{M} + \frac{\alpha^2}{K_{dr}} \right)(p_h^\D)^{k,i}  q_h^\D  
+ \int_{\Dint} K^{\D} \nabla (p_h^\D)^{k,i} \, \nabla q_h^{\D} +  \sum_{j=1}^{\numFrac}\int_{\omega^j} \eta^j \left((p_h^{\D,j})^{k,i}_+\,(q_h^{\D,j})_+ + (p^{\D,j}_h)^{k,i}_-\,(q^{\D,j}_h)_-\right)\\
 -  \sum_{j=1}^{\numFrac}\int_{\omega^j} \eta^j (p_h^{F^j})^{k,i} \left((q_h^{\D,j})_+ + (q_h^{\D,j})_-\right)
= \frac{1}{\Delta t}\left(\int_{\Dint} \frac{1}{M} (p_h^\D)^{k-1}  q_h^\D + \int_{\Dint} \alpha \div \bs{u}_h^{k-1}  q_h^\D\right)+ \int_{\Dint} g\, q_h^\D  
\\
- \frac{1}{\Delta t}\int_{\Dint} \alpha \div \bs{u}_h^{k,i-1}  q_h^\D 
+\frac{1}{\Delta t}\int_{\Dint}\frac{\alpha^2}{K_{dr}}(p_h^\D)^{k,i-1}  q_h^\D 
\quad \forall \; q_h^\D\in V_h^{f,\D} 
\end{aligned}
\\
\label{eq:flow2d_fs}
\begin{aligned}
    \int_{\omega^j} K^{F^j} \nabla (p_h^{F^j})^{k,i} \, \nabla q_h^{F^j} + \int_{\omega^j} 2 \eta^j\, (p_h^{F^j})^{k,i} \, q_h^{F^j} - \int_{\omega^j} \eta^j \left((p_h^{\D,j})^{k,i}_+ + (p^{\D,j}_h)^{k,i}_-\right) q_h^{F^j} = 0  \quad \forall \; q_h^{F^j}\in V_h^{f,F^j} 
    \\ \mbox{and 
} \forall j=1,\ldots,\numFrac
\end{aligned}
    \end{align}
    Then find $\bs{u}_h^{k,i}\in\bs{V}^m_h$ and for each $j=1,\ldots,\numFrac$ $(\hat{\bs{\Lambda}}_h^j)^{k,i}\in \bs{Q}_h^{m,j}$ such that
    \begin{align}
\begin{aligned}
 \int_{\Dint} \bs{C} \bs{\epsilon}(\bs{u}^{k,i}_h) : \bs{\epsilon}(\bs{v}_h) 
 &+ \sum_{j=1}^{\numFrac} \int_{\omega^j} (\hat{\bs{\Lambda}}_h^j)^{k,i}\left((\bs{v}_h^j)_- - (\bs{v}_h^j)_+\right)
 = \int_{\Dint} \alpha (p_h^\D)^{k,i} \div \bs{v}_h 
\quad \forall \; \bs{v}_h\in \bs{V}_h^{m}, 
\end{aligned}
\\
\int_{\omega^j}\left((\bs{u}_h^j)^{k,i}_- - (\bs{u}_h^j)^{k,i}_+\right)\,\bs{\tau}_h^j =0 \quad \forall \; \bs{\tau}_h^j \in \bs{Q}_h^{m,j} ,  \mbox{and 
} \forall j=1,\ldots,\numFrac 
\end{align}
\end{problem}
We now prove the convergence of the fixed-stress splitting scheme for the coupled problem defined above by extending the techniques of \cite{Both2017,Storvik2019} to our setting.
and we define the following error measures:
\begin{equation}\label{def:errorMeasures}
    \begin{aligned}
        \epD&= (p_h^\D)^{k,i}-(p_h^\D)^{k},\;\; 
        \epF= (p_h^{F^j})^{k,i}-(p_h^{F^j})^{k} \; \forall j=1,\ldots,\numFrac,\;\; 
        \\
        \eU  &= \bs{u}_h^{k,i}-\bs{u}_h^{k},\;\;
        \eLambda = (\hat{\bs{\Lambda}}_h^j)^{k,i}-(\hat{\bs{\Lambda}}_h^j)^{k}
    \end{aligned}
\end{equation}

\begin{theorem}
Under Assumptions \ref{assumptionCoeff_omega} - \ref{assumptionCoeff_frac}, let $(p_h^\D)^{k}$, $\bs{u}_h^{k}$, $(\hat{\bs{\Lambda}}_h^j)^{k}$ and $(p_h^{F^j})^{k}$ $\forall j=1,\ldots,\numFrac$ be the solutions of Problem \ref{pb:backEulerDiscr}. 
Moreover, let $(p_h^\D)^{k,i}$, $\bs{u}_h^{k,i}$, $(\hat{\bs{\Lambda}}_h^j)^{k,i}$ and $(p_h^{F^j})^{k,i}$ $\forall j=1,\ldots,\numFrac$ be the solutions of Problem \ref{pb:FixedStress-dec} and we consider the error measures defined in \eqref{def:errorMeasures}.
Then we get the following estimates for a given $\delta_{F}\leq 2$
    \begin{equation}\label{estim:elast}
        2\mu \norm[0,\Dint]{\bs{\epsilon}(\eU)}^2+\lambda \norm[0,\Dint]{\div\eU}^2 \leq  \frac{\alpha^2}{K_{dr}} \norm[0,\Dint]{\epD}^2 \,,
\end{equation}
\begin{equation}\label{estim:lambda}
    \norm[\bs{Q}_h^{m,j}]{\eLambda} ^2
      \leq \frac{4\alpha^2}{\beta^2\,K_{dr}} \norm[0,\Dint]{\epD}^2\,, \quad \forall j=1,\ldots,\numFrac \,,
\end{equation}
\begin{equation}\label{estim:pFrac}
    K^{F^j} \norm[0,\omega^j]{\nabla \epF}^2  + 2 \eta^j\norm[0,\omega^j]{\epF} ^2
   \leq \frac{\eta^j}{\delta_F} \norm[0,\omega^j]{(\epD)_{\pm}^j}^2\,, \quad \forall j=1,\ldots,\numFrac \,,
\end{equation}
 \begin{equation}\label{estim:contraction}
     \left( \frac{1}{M} +\frac{\alpha^2}{2K_{dr}} + \frac{\Delta t K^{\D}}{C_P} \right) \norm[0,\Dint]{\epD}^2  
   \leq \frac{\alpha^2}{2K_{dr}}\norm[0,\Dint]{\epDprec}^2 \,,
 \end{equation}
 that provide the convergence of the fixed-stress splitting scheme, where $\norm[0,\omega^j]{(\epD)_\pm^j}^2:=\int_{\omega^j} \left((\epD)^j_+ + (\epD)^j_-\right)^2$\,.
\end{theorem}
\begin{proof}
Subtracting the equations of Problem \ref{pb:FixedStress-dec} from the corresponding equations of Problem \ref{pb:backEulerDiscr}, we obtain the following error equations
\begin{align} \label{eq:errorpD}
 &   \begin{aligned}
&\int_{\Dint} \frac{1}{M} \epD q_h^\D  
+ \int_{\Dint} \frac{\alpha^2}{K_{dr}} \left( \epD -\epDprec \right)q_h^\D
+ \Delta t\int_{\Dint} K^{\D} \nabla \epD \, \nabla q_h^{\D} 
+ \int_{\Dint} \alpha \div \eUprec  q_h^\D 
\\
&+ \sum_{j=1}^{\numFrac} \Delta t\int_{\omega^j} \eta^j \left((\epD)_+\,(q_h^{\D,j})_+ + (\epD)_-\,(q^{\D,j}_h)_-\right) 
 -  \Delta t \int_{\omega^j} \eta^j \epF \left((q_h^{\D,j})_+ + (q_h^{\D,j})_-\right)
 =0
\quad \forall \; q_h^\D\in V_h^{f,\D} 
\end{aligned}
\\ \label{eq:errorpF}
&\begin{aligned}
    \int_{\omega^j} K^{F^j} \nabla \epF \, \nabla q_h^{F^j} + \int_{\omega^j} 2 \eta^j\, \epF \, q_h^{F^j} - \int_{\omega^j} \eta^j \left((\epD)^j_+ + (\epD)^j_-\right) q_h^{F^j} = 0  \quad \forall \; q_h^{F^j}\in V_h^{f,F^j} 
    \\ \mbox{and 
} \forall j=1,\ldots,\numFrac
\end{aligned}
\\ \label{eq:errorU}
& \int_{\Dint} \bs{C} \bs{\epsilon}(\eU) : \bs{\epsilon}(\bs{v}_h)-\int_{\Dint} \alpha \epD \div \bs{v}_h
 + \sum_{j=1}^{\numFrac} \int_{\omega^j}
 \eLambda
 \left((\bs{v}_h^j)_- - (\bs{v}_h^j)_+\right)
 = 0 
\quad \forall \; \bs{v}_h\in \bs{V}_h^{m}, 
\\ \label{eq:errorLambda}
&\int_{\omega^j}\left((\eU)^j_- - (\eU)^j_+\right)\,\bs{\tau}_h^j =0 \quad \forall \; \bs{\tau}_h^j \in \bs{Q}_h^{m,j} ,  \mbox{and 
} \forall j=1,\ldots,\numFrac
\end{align}
We notice that considering \eqref{eq:errorLambda} for any $j=1,\ldots,\numFrac$ with   
$\bs{\tau}_h^j=\eLambda$,
we get 
\begin{equation}\label{eq:errorLambdaWithTest}
    \begin{aligned}
        \int_{\omega^j}\left((\eU)^j_- - (\eU)^j_+\right)\,\eLambda =0\,.
    \end{aligned}
\end{equation}
Now, considering \eqref{eq:errorU} with $\bs{v}_h=\eU$ together with \eqref{eq:errorLambdaWithTest} and the constitutive relation of linear elasticity \eqref{eq:ElasticityConstitutiveRel}, we derive
\begin{equation}\label{eq:erroreUwitheU}
    2\mu \norm[0,\Dint]{\bs{\epsilon}(\eU)}^2+\lambda \norm[0,\Dint]{\div\eU}^2
    -\int_{\Dint} \alpha \epD \div \eU
 = 0 \,.
\end{equation}
Then, applying the Young's inequality, i.e. given $a,b,\delta\in\mathbb{R}$ we have $ab\leq\frac{a^2}{2\delta}+\frac{\delta b^2}{2}$, with $\delta=K_{dr}$ 
, we obtain
\begin{equation}
    \begin{aligned}
        2\mu \norm[0,\Dint]{\bs{\epsilon}(\eU)}^2+\lambda \norm[0,\Dint]{\div\eU}^2 &\leq \frac{\alpha^2}{2K_{dr}} \norm[0,\Dint]{\epD}^2+\frac{K_{dr}}{2} \norm[0,\Dint]{\div\eU}^2   
        \\
        &\leq \frac{\alpha^2}{K_{dr}} \norm[0,\Dint]{\epD}^2 
    \end{aligned}
\end{equation}
where in the last step we consider \eqref{eq:relationDrainedBulk} with $\bs{v}_h=\eU$, hence obtaining \eqref{estim:elast}.

Now, let us prove the estimate for each $\eLambda$, $\forall j = 1,\ldots, \numFrac$.
We apply the assumption \ref{ass:inf-sup}, the error equation \eqref{eq:errorU}, the Cauchy-Schwarz inequality together with the relation \eqref{eq:relationDrainedBulk}, and the above proved estimate \eqref{estim:elast},  and we get
\begin{equation}
  \begin{aligned}
      &\norm[\bs{Q}_h^{m,j}]{\eLambda} 
      \leq \sum_{j=1}^\numFrac \norm[\bs{Q}_h^{m,j}]{\eLambda}
      \leq \frac{1}{\beta} 
      \sup_{
            \bs{v}_h \in \bs{V}^m_h \colon
            \norm[\bs{V}^m_h]{\bs{v}_h}\neq 0
      } 
      \frac{\sum_{j=1}^{\numFrac}\int_{\omega^j}\left((\bs{v}_h^j)_- - (\bs{v}_h^j)_+\right)\,\eLambda}{\left(2\mu\norm[0,\D]{\bs{\epsilon}(\bs{v}_h)}^2+\lambda\norm[0,\D]{\div\bs{v}_h}^2\right)^{\frac{1}{2}}}
      \\
      &\leq  \frac{1}{\beta} \sup_{\bs{v}_h \in \bs{V}^m_h\colon \norm[\bs{V}^m_h]{\bs{v}_h}\neq 0} \frac{\int_{\Dint} \alpha \epD \div \bs{v}_h-\int_{\Dint} \bs{C} \bs{\epsilon}(\eU) : \bs{\epsilon}(\bs{v}_h)}{\left(2\mu\norm[0,\D]{\bs{\epsilon}(\bs{v}_h)}^2+\lambda\norm[0,\D]{\div\bs{v}_h}^2\right)^{\frac{1}{2}}}
      \\
      &\leq  \frac{1}{\beta} \sup_{\bs{v}_h \in \bs{V}^m_h\colon \norm[\bs{V}^m_h]{\bs{v}_h}\neq 0}
      \frac{(\frac{\alpha}{\sqrt{K_{dr}}}\norm[0,\Dint]{\epD}+ (2\mu\norm[0,\D]{\bs{\epsilon}(\eU)}^2+\lambda\norm[0,\D]{\div\eU}^2)^\frac{1}{2})(2\mu\norm[0,\D]{\bs{\epsilon}(\bs{v}_h)}^2+\lambda\norm[0,\D]{\div\bs{v}_h}^2)^{\frac{1}{2}}}{\left(2\mu\norm[0,\D]{\bs{\epsilon}(\bs{v}_h)}^2+\lambda\norm[0,\D]{\div\bs{v}_h}^2\right)^{\frac{1}{2}}}
      \\
      &\leq  \frac{2\alpha}{\beta\sqrt{K_{dr}}} \norm[0,\Dint]{\epD}\,.
  \end{aligned}  
\end{equation}
Hence, we obtain \ref{estim:lambda}.

Now, let us consider \eqref{eq:errorpF} with $q_h^{F^j}=\epF$, $\forall j=1,\ldots,\numFrac$
\begin{equation}\label{eq:errorpFwithpF}
K^{F^j} \norm[0,\omega^j]{\nabla \epF}^2  + 2 \eta^j\norm[0,\omega^j]{\epF}^2 - \int_{\omega^j} \eta^j \left((\epD)^j_+ + (\epD)^j_-\right) \epF = 0 \,,
\end{equation}
applying the Young’s inequality with $\delta= \delta_{F}$, reminding that $\delta_{F} \leq 2$ we get $\forall j=1,\ldots,\numFrac$
\begin{equation}
\begin{aligned}
    K^{F^j} \norm[0,\omega^j]{\nabla \epF}^2  + 2 \eta^j\norm[0,\omega^j]{\epF}^2 
    &\leq  \frac{\eta^j}{2\delta_F} \norm[0,\omega^j]{(\epD)_\pm^j}^2 + \frac{\eta^j \delta_F}{2} \norm[0,\omega^j]{\epF}^2
       \\
     &  \leq \frac{\eta^j}{\delta_F} \norm[0,\omega^j]{(\epD)_\pm^j}^2\,.
\end{aligned}
\end{equation}
Hence we get \eqref{estim:pFrac}.

Furthermore, let us consider \eqref{eq:errorpD} with $q_h^\D=\epD$
\begin{equation}\label{eq:errorpDwithpD}
    \begin{aligned}
& \frac{1}{M} \norm[0,\Dint]{\epD}^2  
+ \int_{\Dint} \frac{\alpha^2}{K_{dr}} \left( \epD -\epDprec \right)\epD
+ \Delta t K^{\D} \norm[0,\Dint]{\nabla \epD}^2 
+ \int_{\Dint} \alpha \div \eUprec  \epD 
\\
&+ \sum_{j=1}^{\numFrac} \Delta t\int_{\omega^j} \eta^j \left([(\epD)^j_+]^2 
+ [(\epD)^j_-]^2\right) 
 -  \Delta t\int_{\omega^j} \eta^j \epF \left((\epD)^j_+ + (\epD)^j_-\right)
 =0\,.
\end{aligned}
\end{equation}
Then, summing together \eqref{eq:erroreUwitheU}, \eqref{eq:errorpDwithpD}, \eqref{eq:errorpFwithpF} $\forall j=1,\ldots,\numFrac$ multiplied by $\Delta t$, and applying the relation 
\begin{equation*}
\int_{\Dint}\frac{\alpha^2}{K_{dr}}  \left( \epD -\epDprec \right)\epD=\frac{\alpha^2}{2K_{dr}} \left(\norm[0,\Dint]{\epD -\epDprec}^2 +\norm[0,\Dint]{\epD}^2-\norm[0,\Dint]{\epDprec}^2\right)
\end{equation*}
 we get
\begin{equation}\label{eq:3equations-s1}
\begin{aligned}
&
    2\mu \norm[0,\Dint]{\bs{\epsilon}(\eU)}^2+\lambda \norm[0,\Dint]{\div\eU}^2
   + \left( \frac{1}{M} +\frac{\alpha^2}{2K_{dr}} \right) \norm[0,\Dint]{\epD}^2  + \Delta t K^{\D} \norm[0,\Dint]{\nabla \epD}^2
   \\
   &
   +  \sum_{j=1}^{\numFrac}\Delta t\int_{\omega^j} \eta^j \left([(\epD)^j_+]^2 
+ [(\epD)^j_-]^2\right) 
   \\
   &
+  \frac{\alpha^2}{2K_{dr}} \left( \norm[0,\Dint]{\epD -\epDprec}^2 -\norm[0,\Dint]{\epDprec}^2 \right)
- \int_{\Dint} \alpha \div\left(\eU- \eUprec\right)  \epD 
\\
&
+\sum_{j=1}^{\numFrac}\Delta t K^{F^j} \norm[0,\omega^j]{\nabla \epF}^2  + 2 \Delta t\eta^j\norm[0,\omega^j]{\epF}^2-  \int_{\omega^j} 2\Delta t\,\eta^j \epF \left((\epD)^j_+ + (\epD)^j_-\right)
 = 0\,.
 \end{aligned}
\end{equation}
 
We want to manage the term $\int_{\Dint} \alpha \div\left(\eU- \eUprec\right)  \epD$. 
First, let us write the $i-1$-th counterpart of the error equation \eqref{eq:errorLambda}, i.e.
\begin{equation}\label{eq:errorLambdaPrec}
    \int_{\omega^j}\left((\eUprec)^j_- - (\eUprec)^j_+\right)\,\bs{\tau}_h^j =0 \quad \forall \; \bs{\tau}_h^j \in \bs{Q}_h^{m,j} ,  \mbox{and 
} \forall j=1,\ldots,\numFrac
\end{equation} 
 Then, we consider \eqref{eq:errorU} with $\bs{v}_h=\eU- \eUprec$, and we apply the constitutive relation of linear elasticity \eqref{eq:ElasticityConstitutiveRel} together with \eqref{eq:errorLambdaWithTest} and with \eqref{eq:errorLambdaPrec} choosing 
 $\bs{\tau}_h^j=\eLambda$.
 Hence,
 we get
 \begin{equation}
 \begin{aligned}
 \int_{\Dint} \alpha \epD \div (\eU- \eUprec)
&=     \int_{\Dint} \bs{C} \bs{\epsilon}(\eU) : \bs{\epsilon}(\eU- \eUprec)
 + \sum_{j=1}^{\numFrac} \int_{\omega^j} 
 \eLambda
 \left((\eU- \eUprec)^j_- - (\eU- \eUprec)^j_+\right)
 \\
 &
 \begin{aligned}
  =   \int_{\Dint} 2\mu \bs{\epsilon}(\eU) : \bs{\epsilon}(\eU- \eUprec) + 
 \int_{\Dint} \lambda \div \eU \,\div (\eU- \eUprec)
  \,.
 \end{aligned}
 \end{aligned}
 \end{equation}
 Then, we substitute the previous relation in \eqref{eq:3equations-s1} obtaining
 \begin{equation}\label{eq:3equations-s2}
     \begin{aligned}
&
    2\mu \norm[0,\Dint]{\bs{\epsilon}(\eU)}^2+\lambda \norm[0,\Dint]{\div\eU}^2
   + \left( \frac{1}{M} +\frac{\alpha^2}{2K_{dr}} \right) \norm[0,\Dint]{\epD}^2  + \Delta t K^{\D} \norm[0,\Dint]{\nabla \epD}^2
   \\
   &
   +  \sum_{j=1}^{\numFrac}\Delta t\int_{\omega^j} \eta^j \left([(\epD)^j_+]^2 
+ [(\epD)^j_-]^2\right) 
   \\
   &
+  \frac{\alpha^2}{2K_{dr}} \left( \norm[0,\Dint]{\epD -\epDprec}^2 -\norm[0,\Dint]{\epDprec}^2 \right)
-\int_{\Dint} 2\mu \bs{\epsilon}(\eU) : \bs{\epsilon}(\eU- \eUprec) 
 - 
 \int_{\Dint} \lambda \div \eU \,\div (\eU- \eUprec)
\\
&
+\sum_{j=1}^{\numFrac}\Delta t K^{F^j} \norm[0,\omega^j]{\nabla \epF}^2  + 2 \eta^j\Delta t\norm[0,\omega^j]{\epF}^2-  \Delta t\int_{\omega^j} 2\,\eta^j \epF \left((\epD)^j_+ + (\epD)^j_-\right)
 = 0\,.
 \end{aligned}
 \end{equation}
 Then, we move all the negative terms to the right-hand side and bound the terms $\int_{\Dint} 2\mu \bs{\epsilon}(\eU) : \bs{\epsilon}(\eU- \eUprec) 
+
 \int_{\Dint} \lambda \div \eU \,\div (\eU- \eUprec)$  and $\int_{\omega^j} 2\Delta t\eta^j \epF \left((\epD)^j_+ + (\epD)^j_-\right)$ by the Young’s inequality with $\delta = \delta_U\in[1,2]$ and  $\delta = 2$, respectively.  
Hence, we get
 \begin{equation}\label{eq:3equations-s3}
     \begin{aligned}
&
    2\mu \norm[0,\Dint]{\bs{\epsilon}(\eU)}^2+\lambda \norm[0,\Dint]{\div\eU}^2
   + \left( \frac{1}{M} +\frac{\alpha^2}{2K_{dr}} \right) \norm[0,\Dint]{\epD}^2  + \Delta t K^{\D} \norm[0,\Dint]{\nabla \epD}^2
   \\
   &
   +  \sum_{j=1}^{\numFrac}\Delta t\int_{\omega^j} \eta^j \left([(\epD)^j_+]^2 
+ [(\epD)^j_-]^2\right) 
   \\
   &
+  \frac{\alpha^2}{2K_{dr}} \norm[0,\Dint]{\epD -\epDprec}^2  + \sum_{j=1}^{\numFrac}\Delta t K^{F^j}\norm[0,\omega^j]{\nabla \epF}^2  + 2 \Delta t\eta^j\norm[0,\omega^j]{\epF}^2
\\
&\leq \frac{\alpha^2}{2K_{dr}}\norm[0,\Dint]{\epDprec}^2 
+
\frac{\delta_U}{2}\left( 2\mu \norm[0,\Dint]{\bs{\epsilon}(\eU)}^2+\lambda\norm[0,\Dint]{\div\eU}^2\right)
\\
 &
+\frac{1}{2\delta_U}\left( 2\mu \norm[0,\Dint]{\bs{\epsilon}(\eU- \eUprec) }^2+\lambda\norm[0,\Dint]{\div(\eU- \eUprec) }^2\right)
\\
&
+ \sum_{j=1}^{\numFrac} \frac{\Delta t\,\eta^j}{2} 
\int_{\omega^j} \left((\epD)^j_+ + (\epD)^j_-\right)^2
+
2\Delta t\,\eta^j 
\norm[0,\omega^j]{\epF}^2
\,.
 \end{aligned}
 \end{equation}
Notice that for all $j=1,\ldots,\numFrac$ we have the relation
\begin{equation}
    \int_{\omega^j} \left((\epD)^j_+ + (\epD)^j_-\right)^2 \leq 2 \int_{\omega^j}  \left([(\epD)^j_+]^2 
+ [(\epD)^j_-]^2\right) \,,
\end{equation}
and, simplifying terms, \eqref{eq:3equations-s3} results as
  \begin{equation}\label{eq:3equations-s3bis}
     \begin{aligned}
&
    2\mu \norm[0,\Dint]{\bs{\epsilon}(\eU)}^2+\lambda \norm[0,\Dint]{\div\eU}^2
   + \left( \frac{1}{M} +\frac{\alpha^2}{2K_{dr}} \right) \norm[0,\Dint]{\epD}^2  + \Delta t K^{\D} \norm[0,\Dint]{\nabla \epD}^2
   \\
   &
+  \frac{\alpha^2}{2K_{dr}} \norm[0,\Dint]{\epD -\epDprec}^2  + \sum_{j=1}^{\numFrac}\Delta t K^{F^j}\norm[0,\omega^j]{\nabla \epF}^2 
\\
&\leq \frac{\alpha^2}{2K_{dr}}\norm[0,\Dint]{\epDprec}^2 
+
\frac{\delta_U}{2}\left( 2\mu \norm[0,\Dint]{\bs{\epsilon}(\eU)}^2+\lambda\norm[0,\Dint]{\div\eU}^2\right)
\\
&
+\frac{1}{2\delta_U}\left( 2\mu \norm[0,\Dint]{\bs{\epsilon}(\eU- \eUprec) }^2+\lambda\norm[0,\Dint]{\div(\eU- \eUprec) }^2\right)
\,.
 \end{aligned}
 \end{equation}
  To take care of the term involving $\eU-\eUprec$, we consider \eqref{eq:errorU}, subtracting iteration $i-1$ from iteration $i$, then we set the test function $\bs{v}_h=\eU-\eUprec$ and we obtain
\begin{multline}
        2\mu\norm[0,\Dint]{\bs{\epsilon}(\eU-\eUprec)}^2 + \lambda\norm[0,\Dint]{\div(\eU-\eUprec)}^2 
      \\
        + \sum_{j=1}^{\numFrac} \int_{\omega^j} 
      (\eLambda-\eLambdaprec)
        \left((\eU- \eUprec)^j_- - (\eU- \eUprec)^j_+\right)
        =\int_{\Dint} \alpha (\epD-\epDprec) \div (\eU-\eUprec)
 \,.
    \end{multline}
  Considering equations \eqref{eq:errorLambda} and \eqref{eq:errorLambdaPrec}, both with 
  $\bs{\tau}_h^j=(\eLambda-\eLambdaprec)$,
  the term on the fractures disappears, hence we get
 \begin{equation}
 \begin{aligned}
     2\mu\norm[0,\Dint]{\bs{\epsilon}(\eU-\eUprec)}^2 + \lambda\norm[0,\Dint]{\div(\eU-\eUprec)}^2 =\int_{\Dint} \alpha (\epD-\epDprec) \div (\eU-\eUprec)
 \,.
 \end{aligned}
 \end{equation}
 Then, we apply the Cauchy-Schwarz inequality to the right hand-side, together with \eqref{eq:relationDrainedBulk} with $\bs{v}_h=\eU-\eUprec$, 
and we obtain
 \begin{multline}\label{eq:estimOfEUmEUPREC}
         2\mu\norm[0,\Dint]{\bs{\epsilon}(\eU-\eUprec)}^2 + \lambda\norm[0,\Dint]{\div(\eU-\eUprec)}^2   \\
         \leq \alpha \norm[0,\Dint]{\epD-\epDprec}\norm[0,\Dint]{\div (\eU-\eUprec)}
         \leq \frac{\alpha^2}{K_{dr}} \norm[0,\Dint]{\epD-\epDprec}^2
 \end{multline}
 Then, substituting \eqref{eq:estimOfEUmEUPREC} in \eqref{eq:3equations-s3bis}, simplifying and using that $\delta_U\geq 1$, we get
 \begin{multline}\label{eq:estim-global}
    (1-\frac{\delta_U}{2})\left(2\mu \norm[0,\Dint]{\bs{\epsilon}(\eU)}^2+\lambda \norm[0,\Dint]{\div\eU}^2\right)
   + \left( \frac{1}{M} +\frac{\alpha^2}{2K_{dr}} \right) \norm[0,\Dint]{\epD}^2  + \Delta t K^{\D} \norm[0,\Dint]{\nabla \epD}^2
   \\
+  \sum_{j=1}^{\numFrac} \Delta t K^{F^j}\norm[0,\omega^j]{\nabla \epF}^2  
\leq \frac{\alpha^2}{2K_{dr}}\norm[0,\Dint]{\epDprec}^2 
          \,.
 \end{multline}
 This implies, using that  $\delta_U\leq 2$ and $1\leq \delta_P\leq 2$ and the Poincaré inequality, that we get the estimate
 \begin{equation}
     \left( \frac{1}{M} +\frac{\alpha^2}{2K_{dr}} + \frac{\Delta t K^{\D}}{C_P} \right) \norm[0,\Dint]{\epD}^2  
   \leq \frac{\alpha^2}{2K_{dr}}\norm[0,\Dint]{\epDprec}^2 
 \end{equation}
 that implies that we have convergence of the scheme.
\end{proof}
\begin{remark}
   The previous theorem naturally extends to the case where each fracture is associated with its own parameter $\delta_F^j\leq 2$ for each $j=1,\ldots,\numFrac$. Under this assumption, the estimate \eqref{estim:pFrac} holds independently on each fracture, with $\delta_F^j$ replacing $\delta_F$.
\end{remark}
Finally, since the application of the fixed-stress splitting scheme described in Problem \ref{pb:FixedStress-dec} decoupled the flow and the mechanics subproblems, in the following we address independently the two problems.

\section{Optimization based domain decomposition for the 2D-3D flow equation}\label{sec:FlowProblem}

Thanks to the use of the fixed-stress split, we are now left with the resolution of two problems. Here we consider the flow subproblem, which, as said, is a mixed-dimensional coupled system of equations, described by equations \eqref{eq:flow3d_fs} and \eqref{eq:flow2d_fs}. We describe here the optimization based domain decomposition strategy introduced in \cite{BPSa} and further developed in \cite{S2024Finel} for discrete fracture and matrix problems. The equations on the 3D and 2D domains are decoupled through the introduction of three additional interface variables for each fracture, namely, $\Psi^{+^j},\Psi^{-^j},\Psi^{F^j}$ for any $j=1,\ldots,\numFrac$, that represent the trace of the 3D pressure on $\omega^j_+$, $\omega^j_-$ and the 2D pressure on $F^j$, respectively. The matching between the new interface variables and the original ones is obtained through the minimization of a functional, as described in the following. 

Having $\star \in \{+,-,F\}$, let $W_h^{f,\star^j}$ denote the discrete space of $\Psi^{\star^j}$ for $j=1,\ldots,\numFrac$. 
It should be noted that the discrete space $W_h^{f,\star^j}$ can be chosen independently for each variable, and also independently of the mesh of the 3D domain and of the fractures. This gives great flexibility and robustness to the approach.
We refer to \cite{S2024Finel} for more details.

It is possible to rewrite equations \eqref{eq:flow3d_fs} and \eqref{eq:flow2d_fs} as the following problem.
\begin{problem}\label{pb:optFlux}
    Find $(p_h^\D)^{k,i}\in V_h^{f,\D} $, $(p_h^{F^j})^{k,i}\in V_h^{f,F^j}$, $\Psi^{+^j}\in W_h^{f,+^j}$, $\Psi^{-^j}\in W_h^{f,-^j}$, $\Psi^{F^j}\in W_h^{f,F^j}$ for any $j=1,\ldots,\numFrac$ such that
    \begin{equation}
        \min \sum_{j=1}^{\numFrac}\norm[\lebl{\omega^j}]{(p_h^{\D,j})_+^{k,i}-\Psi^{+^j}}^2 + \norm[\lebl{\omega^j}]{(p_h^{\D,j})_-^{k,i}-\Psi^{-^j}}^2 + \norm[\lebl{\omega^j}]{(p_h^{F^j})^{k,i}-\Psi^{F^j}}^2
    \end{equation}
    subject to the following PDEs
    \begin{align}
     \begin{aligned}
\frac{1}{\Delta t}\int_{\Dint} \left(\frac{1}{M} + \frac{\alpha^2}{K_{dr}} \right)(p_h^\D)^{k,i}  q_h^\D  
+ \int_{\Dint} K^{\D} \nabla (p_h^\D)^{k,i} \, \nabla q_h^{\D} +  \sum_{j=1}^{\numFrac}\int_{\omega^j} \eta^j \left((p_h^{\D,j})^{k,i}_+\,(q_h^{\D,j})_+ + (p^{\D,j}_h)^{k,i}_-\,(q^{\D,j}_h)_-\right)\\
 -  \sum_{j=1}^{\numFrac}\int_{\omega^j} \eta^j \Psi^{F^j} \left((q_h^{\D,j})_+ + (q_h^{\D,j})_-\right)
= \frac{1}{\Delta t}\left(\int_{\Dint} \frac{1}{M} (p_h^\D)^{k-1}  q_h^\D + \int_{\Dint} \alpha \div \bs{u}_h^{k-1}  q_h^\D\right)+ \int_{\Dint} g\, q_h^\D  
\\
- \frac{1}{\Delta t}\int_{\Dint} \alpha \div \bs{u}_h^{k,i-1}  q_h^\D 
+\frac{1}{\Delta t}\int_{\Dint}\frac{\alpha^2}{K_{dr}}(p_h^\D)^{k,i-1}  q_h^\D 
\quad \forall \; q_h^\D\in V_h^{f,\D} 
\end{aligned}\label{eq:constraintsDomain}
\\
\begin{aligned}
    \int_{\omega^j} K^{F^j} \nabla (p_h^{F^j})^{k,i} \, \nabla q_h^{F^j} + \int_{\omega^j} 2 \eta^j\, (p_h^{F^j})^{k,i} \, q_h^{F^j} - \int_{\omega^j} \eta^j\left(\Psi^{+^j} + \Psi^{-^j}\right) q_h^{F^j} = 0  \quad \forall \; q_h^{F^j}\in V_h^{f,F^j} 
    \\ \mbox{and 
} \forall j=1,\ldots,\numFrac
\end{aligned}\label{eq:constraintsFrac}
    \end{align}
\end{problem}

Now, we give the construction of the matrix representation of the previous problem.
Let $\{\varphi_l^\D\}_{l=1}^{\mathcal{N}^\D}$ be a set of basis functions of $V_h^{f,\D}$ with $\mathcal{N}^\D$ being its dimension and, for each fracture $F^j$ with $j=1,\ldots, \numFrac$, let $\{\varphi_l^{F^j}\}_{l=1}^{\mathcal{N}^{F^j}}$ and $\{\vartheta^{\star^j}_l\}_{l=1}^{\mathcal{\tilde{N}}^{\star^j}}$, for any $\star\in\{+,-,F\}$, be the sets of basis functions of $V_h^{f,F^j}$ and $W_h^{f,\star^j}$, with $\mathcal{N}^{F^j}$ and $\mathcal{\tilde{N}}^{\star^j}$ being the dimensions, respectively.
Hence, for any step of the proposed method indexed by the couple $(k,i)$, the discrete variables are written as
\begin{equation}
\begin{aligned}
(p_h^\D)^{k,i} = \sum_{l=1}^{\mathcal{N}^\D} \limits \mathrm{p}_l^\D \varphi_l^\D,\;\; 
\mbox{and for any }j=1,\ldots, \numFrac\;\;
(p_h^{F^j})^{k,i} = \sum_{l=1}^{\mathcal{N}^{F^j}}\limits \mathrm{p}^{F^j}_l\varphi_l^{F^j},
\\
\Psi^{+^j} = \sum_{l=1}^{\mathcal{\tilde{N}}^{+^j}}\limits
\Psi^{+^j}_l\vartheta_l^{+^j},
\;\;
\Psi^{-^j}= \sum_{l=1}^{\mathcal{\tilde{N}}^{-^j}}\limits
\Psi^{-^j}_l\vartheta_l^{-^j},
\;\;
\Psi^{F^j}= \sum_{l=1}^{\mathcal{\tilde{N}}^{F^j}}\limits
\Psi^{F^j}_l\vartheta_l^{F^j}.
\end{aligned}
\end{equation}
Let $\boldsymbol{G}^\D\in \mathbb{R}^{\mathcal{N}^\D\times \mathcal{N}^\D}$ 
be the matrix defined as
\begin{equation}
\left(\boldsymbol{G}^\D\right)_{l,m} =\sum_{j=1}^\numFrac \int_{\omega^j} ( \varphi_l^{\D,j})_+ (\varphi_m^{\D,j})_+ + ( \varphi_l^{\D,j})_- (\varphi_m^{\D,j})_-\,,
\end{equation}
and for each $j=1,\ldots, \numFrac$, let $\boldsymbol{G}^{F^j}\in\mathbb{R}^{\mathcal{N}^{F^j}\times \mathcal{N}^{F^j}}$ and  $\boldsymbol{G}^{\Psi^{\star^j}}\in\mathbb{R}^{\mathcal{\tilde{N}}^{\star^j}\times \mathcal{\tilde{N}}^{\star^j}}$ for any $\star\in\{+,-,F\}$ be the matrices defined as
\begin{equation}
\left(\boldsymbol{G}^{F^j}\right)_{l,m} = \int_{F^j}  \varphi_l^{F^j} \varphi_m^{F^j} \,,
\;\;
\left(\boldsymbol{G}^{\Psi^{\star^j}}\right)_{l,m} = \int_{F^j} \vartheta_l^{\star^j}\vartheta_m^{\star^j}\,.
\end{equation}
Moreover, for each $j=1,\ldots, \numFrac$ let $\boldsymbol{E}^{+^j}\in\mathbb{R}^{\mathcal{N}^\D\times\mathcal{\tilde{N}}^{+^j} }$, $\boldsymbol{E}^{-^j}\in\mathbb{R}^{\mathcal{N}^\D\times\mathcal{\tilde{N}}^{-^j}}$ and $\boldsymbol{E}^{F^j}\in\mathbb{R}^{\mathcal{N}^{F^j}\times\mathcal{\tilde{N}}^{F^j}}$ be the matrices defined as
\begin{equation}
    \left(\boldsymbol{E}^{+^j}\right)_{l,m} = -\eta\int_{\omega^j}  ( \varphi_l^\D)_+ \vartheta_m^{+^j}\,, \;
    \left(\boldsymbol{E}^{-^j}\right)_{l,m} = -\eta\int_{\omega^j} ( \varphi_l^\D)_- \vartheta_m^{-^j}\,, \;
    \left(\boldsymbol{E}^{F^j}\right)_{l,m} = -\eta\int_{\omega^j} \varphi_l^{F^j} \vartheta_m^{F^j}\,.
\end{equation}

Considering all the fractures, we collect the above defined matrices in the following forms
\begin{equation}
    \begin{aligned}
        \mathbf{G^F} = \begin{bmatrix} \mathbf{G}^{F^1} & 0  & 0
        \\
        0 & \ddots & 0 
        \\
        0 & 0 & \mathbf{G}^{F^{\numFrac}}
        \end{bmatrix}\,,
        \;\;
        \mathbf{G^{\Psi^\star}} = \begin{bmatrix} \mathbf{G}^{\Psi^{\star^1}} & 0  & 0
        \\
        0 & \ddots & 0 
        \\
        0 & 0 & \mathbf{G}^{\Psi^{\star^{\numFrac}}}
        \end{bmatrix} \,,
        \;\;
        \mathbf{E^{\star}} = \begin{bmatrix}\mathbf{E}^{\star^1} & \dots & \mathbf{E}^{\star^{\numFrac}}
        \end{bmatrix}
    \end{aligned}
\end{equation}
with $\star\in\{+,-,F\}$.
The matrix form of the discrete constraint equations \eqref{eq:constraintsDomain} and \eqref{eq:constraintsFrac} can be obtained defining $\mathbf{A}^\D\in\mathbb{R}^{\mathcal{N}^\D\times \mathcal{N}^\D}$ as
\begin{equation}
\begin{aligned} 
(\mathbf{A}^\D)_{l,m} = \frac{1}{\Delta t}\int_{\Dint} \left(\frac{1}{M} + \frac{\alpha^2}{K_{dr}} \right)\varphi_l^\D \varphi_m^\D  
+ \int_{\Dint} K^{\D} \nabla \varphi_l^\D \nabla \varphi_m^\D + 
\\
\sum_{j=1}^{\numFrac}\int_{\omega^j} \eta^j \left(\varphi_l^{\D,j})_+\,(\varphi_m^{\D,j})_+ + (\varphi_l^{\D,j})_-\,(\varphi_m^{\D,j})_-\right)
\end{aligned}
\end{equation}
and for each fracture $j=1,\ldots,\numFrac$ the matrix $\mathbf{A}^{F^j}\in\mathbb{R}^{\mathcal{N}^{F^j}\times\mathcal{N}^{F^j}}$ as
\begin{equation}
    (\mathbf{A}^{F^j})_{l,m} = \int_{\omega^j} K^{F^j} \nabla \varphi_l^{F^j} \, \nabla \varphi_m^{F^j} + \int_{\omega^j} 2 \eta^j\, \varphi_l^{F^j} \, \varphi_m^{F^j}
\end{equation}
and matrices $\mathbf{B}^{\star^j}\in\mathbb{R}^{ \mathcal{N}^{F^j}\times\mathcal{\tilde{N}}^{\star^j}}$ with $\star\in\{+,-\}$ and $\mathbf{B}^{F^j}\in \mathbb{R}^{ \mathcal{N}^{\D}\times\mathcal{\tilde{N}}^{F^j}}$ respectively as
\begin{equation}
    \begin{aligned}
        (\mathbf{B}^{\star^j})_{l,m} =  \int_{\omega^j} \eta^j \vartheta_m^{\star^j} \varphi_l^{F^j}
        \,,
        \quad
        (\mathbf{B}^{F^j})_{l,m} =  \int_{\omega^j} \eta^j \vartheta_m^{F^j} \left((\varphi_l^{\D,j})_+ + (\varphi_l^{\D,j})_-\right)\,.
    \end{aligned}
\end{equation}
Considering all the fractures, we collect the above defined matrices in the following forms
\begin{equation}
\mathbf{A^F} = \mathrm{diag}\left(\mathbf{A}^{F^1}, \dots,   \mathbf{A}^{F^{\numFrac}}\right)\,, \quad
    \mathbf{B^{\star}} = \mathrm{diag}\left(\mathbf{B}^{\star^1}, \dots, \mathbf{B}^{\star^{\numFrac}}\right)
    \,,
    \quad\mbox{with }
    \star\in\{+,-\}\,,
    \end{equation}
    \begin{equation}
    \mathbf{B^F} = \begin{bmatrix} \mathbf{B}^{F^1} & \dots & \mathbf{B}^{F^{\numFrac}}\end{bmatrix}
    \,.
\end{equation}
Finally, taking into account all the matrices defined above, we conclude that solving Problem \ref{pb:optFlux} is equivalent to solving the linear system: 
\begin{equation}
\begin{bmatrix}
    \boldsymbol{\mathrm{G}} & \boldsymbol{\mathrm{C}}^\intercal \\ \boldsymbol{\mathrm{C}} & \boldsymbol{\mathrm{0}}
\end{bmatrix} 
\begin{bmatrix}
    \bm{\chi}\\
    \bm{\varsigma}
\end{bmatrix}=
\begin{bmatrix}
    \bm{0}\\
    \bm{\tilde{g}}
\end{bmatrix}
\end{equation}
where, $\bm{\tilde{g}}$ accounts for all known terms, and, the arrays of unknown coefficients are
\begin{displaymath}
\bm{\chi}=\begin{bmatrix}
    \bm{p}\\
    \bm{\Psi}
\end{bmatrix},
\quad
\bm{p}=
    \begin{bmatrix}
      \bs{\mathrm{p}}^\D\\
      \bs{\mathrm{p}}^{F^1}\\
      \vdots\\
      \bs{\mathrm{p}}^{F^{\numFrac}}
    \end{bmatrix},
    \quad
    \bm{\mathrm{\Psi}}=\begin{bmatrix}
        \bs{{\Psi}}^+\\
        \bs{{\Psi}}^-\\
        \bs{{\Psi}}^F
    \end{bmatrix},
    \quad
    \bs{{\Psi}}^\star=\begin{bmatrix}
        {\Psi}^{\star^1}\\
        \vdots\\
        \Psi^{\star^\numFrac}
    \end{bmatrix}, \ \text{for} \ \star\in\{+,-,F\}, 
\end{displaymath}
and
\begin{equation}
    \boldsymbol{\mathrm{G}} = \begin{bmatrix}
    \boldsymbol{G}^\D & \boldsymbol{0} & \mathbf{E^+} & \mathbf{E^-} & \boldsymbol{0}
    \\
    \boldsymbol{0} & \mathbf{G^F} & \boldsymbol{0} &\boldsymbol{0} & \mathbf{E^F}
    \\
    (\mathbf{E^+})^\intercal & \boldsymbol{0} & \mathbf{G^{\Psi^+}}& \boldsymbol{0}& \boldsymbol{0}
    \\
    (\mathbf{E^-})^\intercal & \boldsymbol{0} &\boldsymbol{0}&\mathbf{G^{\Psi^-}}& \boldsymbol{0}
    \\
    \boldsymbol{0}& (\mathbf{E^F})^\intercal& \boldsymbol{0}& \boldsymbol{0} & \mathbf{G^{\Psi^F}}
    \end{bmatrix},
    \quad    
    \boldsymbol{\mathrm{C}} = \begin{bmatrix}
        \mathbf{A}^\D& \boldsymbol{0}& \boldsymbol{0} &\boldsymbol{0}& - \mathbf{B^{F}}
        \\
        \boldsymbol{0}&  \mathbf{A^F} & - \mathbf{B^{+}} &  - \mathbf{B^{-}} & \boldsymbol{0}
    \end{bmatrix}\,.
\end{equation}
Well posedness of this saddle point problem is shown in \cite{S2024Finel}.

It is however worth mentioning here that the above optimization problem can be solved via a conjugate gradient method. We define:
\begin{displaymath}
    \bm{\mathcal{A}}=\begin{bmatrix}
        \mathbf{A}^\D& \boldsymbol{0}\\
        \boldsymbol{0}& \mathbf{A^F}
    \end{bmatrix}, \quad
    \bm{\mathcal{B}}=\begin{bmatrix}
        \bm{0} & \bm{0} & -\mathbf{B^{F}}\\
        -\mathbf{B^{+}} &  - \mathbf{B^{-}} & \boldsymbol{0}
    \end{bmatrix}, \quad 
    \bm{\mathcal{G}^p}=\begin{bmatrix}
        \boldsymbol{G}^\D & \boldsymbol{0} \\
        \bm{0} & \mathbf{G^F}
    \end{bmatrix},
    \end{displaymath}
    \begin{displaymath}
    \bm{\mathcal{G}^\Psi}=\begin{bmatrix}
        \mathbf{G^{\Psi^+}}& \boldsymbol{0}& \boldsymbol{0}\\
        \boldsymbol{0}& \mathbf{G^{\Psi^-}}& \boldsymbol{0}\\
        \boldsymbol{0}& \boldsymbol{0} & \mathbf{G^{\Psi^F}}
    \end{bmatrix},
    \quad
    \bm{\mathcal{E}}=\begin{bmatrix}
        \mathbf{E^+} & \mathbf{E^-} & \boldsymbol{0}\\
        \boldsymbol{0} & \boldsymbol{0} & \mathbf{E^F}
    \end{bmatrix}
\end{displaymath}
the discrete minimization problem becomes:
\begin{align}
    \min_{\bm{p},\bm{\Psi}} \bm{p}^T \bm{\mathcal{G}^p} \bm{p} + \bm{p}^T \bm{\mathcal{E}} \bm{\Psi} + \bm{\Psi}^T\bm{\mathcal{E}}^T \bm{p} + \bm{\Psi}^T \bm{\mathcal{G}^\psi} \bm{\Psi} \\
    \text{such that} \\
    \bm{\mathcal{A}} \bm{p} + \bm{\mathcal{B}}\bm{\Psi}=\bm{\tilde{g}}
\end{align}
which can be restated as:
\begin{displaymath}
    \min_{\bm{\Psi}} \bm{\Psi}^T \left(\bm{\mathcal{B}}^T \bm{\mathcal{A}}^{-T}\bm{\mathcal{G}^p}\bm{\mathcal{A}}^{-1}\bm{\mathcal{B}}+\bm{\mathcal{B}}^T \bm{\mathcal{A}}^{-T}\bm{\mathcal{E}}+\bm{\mathcal{G}^\Psi}\right)\bm{\Psi} +2\bm{\Psi}^T \left(\bm{\mathcal{B}}^T+\bm{\mathcal{A}}^{-T}\bm{\mathcal{G}^p}\bm{\mathcal{A}}^{-1}+\bm{\mathcal{E}}^T\bm{\mathcal{A}}^{-1}\right)\bm{\tilde{g}}.
\end{displaymath}
This unconstrained minimization problem is solved via the conjugate gradient scheme, with gradient direction $\bm{d}$ given by:
\begin{displaymath}
\bm{d}=-\bm{\mathcal{B}}^T\bm{\varrho}+\bm{\mathcal{E}}^T\bm{p}+\bm{\mathcal{G}^\Psi}\bm{\Psi}, \quad \bm{p}=\bm{\mathcal{A}}^{-1}\left(\bm{\tilde{g}}-\bm{\mathcal{B}}\bm{\Psi}\right), \quad \text{and} \quad \bm{\varrho}=\bm{\mathcal{A}}^{-T}\left(\bm{\mathcal{G}^p}\bm{p}+\bm{\mathcal{E}}\bm{\Psi}\right),
\end{displaymath}
in which the decoupling of the problem on the bulk domain from that on the fractures is given by the block diagonal structure of matrix $\bm{\mathcal{A}}$.
\section{The mechanics problem}\label{sec:mechProblem}

In this section, we briefly recall the formulation of the mechanics subproblem of Problem \ref{pb:FixedStress-dec}, and provide its matrix representation. 
Let $\{\bs{\varphi}^M_l\}_{l=1}^{\mathcal{N}^M}$ be a set of basis functions of $\bs{V}^m_h$ with $\mathcal{N}^M$ being its dimension and, for each fracture $F^j$ with $j=1,\ldots,\numFrac$, let $\{\bs{\ell}_l^j\}_{l=1}^{\mathcal{N}^{Q^j}}$ be the set of basis functions of $\bs{Q}^{m,j}_h$ with $\mathcal{N}^{Q^j}$ being its dimension.
Hence, for any step of the proposed method indexed by the pair $(k,i)$, the discrete variables are written as
\begin{equation}
    \bs{u}^{k,i}_h = \sum_{l=1}^{\mathcal{N}^M} \mathrm{u}_l \bs{\varphi}^M_l\,,
    \quad 
    \mbox{and for any }
    j=1,\ldots,\numFrac \;\; (\hat{\bs{\Lambda}}^j_h)^{k,i} = \sum_{l=1}^{\mathcal{N}^{Q^j}} \hat{\mathrm{\Lambda}}^j_l \bs{\ell}_l^j\,.
\end{equation}
Let $\bs{G}^M\in\mathbb{R}^{\mathcal{N}^M\times\mathcal{N}^M}$ be the matrix defined as
\begin{equation}
    (\bs{G}^M)_{l,m} = \int_{\Dint} \bs{C} \bs{\epsilon}(\bs{\varphi}^M_m) : \bs{\epsilon}(\bs{\varphi}^M_l) 
\end{equation}
and for each fracture $j=1,\ldots,\numFrac$ let $\mathbf{B^M}^j\in\mathbb{R}^{\mathcal{N}^M\times\mathcal{N}^{Q^j}}$ be the matrix defined as
\begin{equation}
    (\mathbf{B^M}^j)_{l,m} = \int_{\omega^j} (\bs{\ell}^j_m)\left((\bs{\varphi}^M_l)_- - (\bs{\varphi}^{M}_l)_+\right)\,.
\end{equation}
Considering all the fractures, we collect the matrices defined above as follows:
\begin{equation}
    \mathbf{B^M} = \begin{bmatrix}\mathbf{B}^{M^1} & \dots & \mathbf{B}^{M^{\numFrac}}\end{bmatrix}\,.
\end{equation}
Finally, it is easy to see that solving the mechanics subproblem of Problem \ref{pb:FixedStress-dec} is equivalent to solving a linear system:
$$\begin{bmatrix}
    \boldsymbol{\mathrm{G}^M} & \boldsymbol{\mathrm{B^M}} \\ \boldsymbol{\mathrm{B^M}}^T & \boldsymbol{\mathrm{0}}
\end{bmatrix} 
\begin{bmatrix}
    \bs{\mathrm{u}}\\
    \bs{\hat{\mathrm{\Lambda}}}
\end{bmatrix}=\begin{bmatrix}
    \bs{\varsigma}\\ \boldsymbol{\mathrm{0}}
\end{bmatrix}\,,$$
in which:
\begin{displaymath}
    \bs{\varsigma}_j=\int_{\Dint} \alpha (p_h^\D)^{k,i} \div \bs{\varphi}^M_j, \quad j=1,\ldots,\mathcal{N}^M\,.
\end{displaymath}
This formulation is analogous to the one in \cite{JHA2014}, with the difference that here the mesh does not conform to the fractures, and the XFEM is used to reproduce discontinuous solutions. Unlike the flow subproblem, the mechanics subproblem does not involve the solution of governing equations on the lower-dimensional fracture domains.

\section{Numerical results}\label{sec:numres}
In this section, we present three numerical tests to assess the viability and the performance of the proposed approach for solving coupled flow and poromechanics problems.
The first test is the uniaxial compaction test, also known as Terzaghi problem \cite{terzaghi,Kadeethum2020}, which is used to assess the accuracy of the fluid-solid coupling. Then, the second test considers a fractured domain, with a single fracture, and compares the solution obtained with the proposed approach with that of an analogous equi-dimensional problem. The third test investigates the behavior of the method in a more complex domain containing four fractures.
All the results presented are obtained using an in-house MATLAB code.

For the spatial discretization we construct a tetrahedral discretization of the bulk domain for the 3D displacement field and pressure, and triangular discretizations on each fracture for the 2D pressure, for the Lagrange multiplier $\bm{\Lambda}_h$ and for the pressure-interface variables $\bm{\Psi}^\star$, $\star\in\{+,-,F\}$.  We recall that the discretization of the fractures is independent of the discretization of the bulk domain, and also the discretization of the 2D pressure can be independent of that of the interface variables and of the Lagrange multipliers. However, in the examples reported here, the same triangular mesh is used for the 2D pressures and for the pressure interface variables, while a possibly different mesh is used for the Lagrange multiplier $\bm{\Lambda}_h$.
The mesh parameter $h$ is thus expressed as a three-value tuple, $h=[h_{\D}, \, h_{F}, \, h_{\Lambda}]$, whose entries denote the maximum element measure of the 3D bulk mesh, the 2D mesh used for the pressure and pressure interface variables, and the 2D mesh used for the Lagrange multipliers, respectively.
The mesh size $h_{\Lambda}$ associated with the Lagrange multipliers is chosen sufficiently large 
to satisfy the inf-sup condition \cite{Aagaard2013}. Linear Lagrangian finite elements are used for pressures and displacements, with additional XFEM basis functions used to represent discontinuous solutions on nonconforming meshes. Here, the displacement field is allowed to be discontinuous, but a continuity condition is introduced in order to obtain the equilibrium stress field on the fractures. The pressure interface variables $\bm{\Psi}^\star$, $\star\in\{+,-,F\}$ and the Lagrange multipliers $\bm{\Lambda}_h$ are approximated using piecewise constant basis functions.

For the time discretization we use the backward Euler scheme in time. At each time step, the mechanics and flow subproblem are decoupled using the fixed-stress scheme. At each fixed-stress iteration the mechanics subproblem is solved monolithically, whereas the mixed-dimensional flow problem is solved via the optimization formulation using the conjugate gradient scheme. The stopping criterion for the fixed stress scheme is a tolerance, denoted by $\mathrm{tol}_{\mathrm{FS}}$ on the relative distance between two consecutive solutions. The same tolerance is used for both the displacement and pressure solutions. The CG iterations are stopped when the residual, normalized by the initial residual, falls below the tolerance $\mathrm{tol}_{\mathrm{CG}}$, on the residual relative to the initial residual. 
The choice of a monolithic scheme for the mechanics problem is not intended to be computationally optimal and is adopted here because the mechanics subproblem is not mixed-dimensional, and its efficient solution is therefore not the focus of this work. A suitable solver for the mechanics subproblem with this formulation is described in \cite{Aagaard2013}.

\subsection{Test 1} 
In this test, we solve the Terzaghi problem \cite{terzaghi,Kadeethum2020} in a homogeneous material, considering a three dimensional domain $\D=[0,1]^3$. The domain is subjected to a compression load $\bs{\sigma}\cdot\bs{n}=-1\,\mathrm{kPa}$, imposed on the face $z=1$, and displacement is allowed along the $z$-axis on the lateral faces. A graphical representation of the domain with boundary conditions for the mechanics problem is depicted in Figure \ref{fig:DataTest1}. 
For the flow problem we impose homogeneous Dirichlet boundary conditions on the top face and homogeneous Neumann conditions on the rest of $\partial\D$.
The exact pressure solution, depending on the time $t$ and the distance $\mathrm{d}:=1-z$ from the drainage boundary, is given by
\begin{equation}\label{eq:exactSolTerzaghi}
    p^\D(\mathrm{d},t) = \sum_{m=0}^{\infty} \frac{4}{\pi(2m+1)} \sin\left(\frac{(2m+1)\pi}{2}\mathrm{d}\right) \exp{\left(-\frac{(2m+1)^2\pi^2}{4}C_v t\right) }\,,
\end{equation}
where
$
    C_v = 3E\frac{1-\nu}{1+\nu} K^\D
$
with  $E$ and $\nu$ denoting respectively the bulk modulus and the Poisson ratio.
We recall the following relations to compute the Lamé coefficients given the bulk modulus  and the Poisson ratio
\begin{equation}
    \lambda = \frac{3E\nu}{1+\nu}\,, \quad \mu = \frac{3E(1-2\nu)}{2(1+\nu)}\,.
\end{equation}
We use the following values for the relevant parameters: bulk modulus $E=10^3\,\mathrm{kPa}$, Poisson ratio $\nu=0.25$, compressibility constant $\frac{1}{M}=0\,\mathrm{Pa}^{-1}$, hydraulic permeability $k^\D = 10^{-12}\,\mathrm{m}^2$, $\mu_f=10^{-6} \,\mathrm{kPa\, s}$, Biot-Willis constant $\alpha = 1$ and a null $g$. 
We consider null initial conditions for pressure $p^\D$ and displacement $\bs{u}$.
We construct the discretization of $\D$ using TetGen \cite{Tetgen2015} setting  $h_{\D}=10^{-4}$, the time-step increment $\Delta t$ for the Backward Euler scheme is taken equal to $1$ and the fixed-stress scheme is terminated when the iterate-based criterion reaches a tolerance of $\mathrm{tol}_{\mathrm{FS}}= 10^{-6}$.

In Figure \ref{fig:resTest1}, we report the results at four distinct time instants $t\in \{5,10,15,20\}$ comparing a one dimensional section of the discrete solution $p^\D_h$ obtained with the proposed routine against the analytical solution \eqref{eq:exactSolTerzaghi}. The results show a strong agreement between the discrete and analytical solutions.
\begin{figure}
     \begin{subfigure}[t]{0.4\textwidth}
         \centering
         \includegraphics[width=0.8\textwidth]{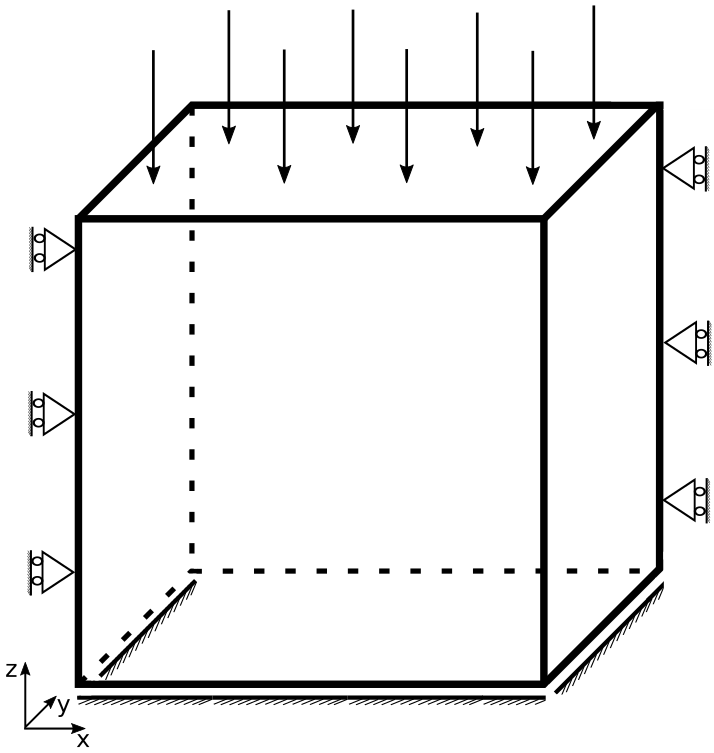}
         \caption{\footnotesize{Domain with boundary conditions}}
         \label{fig:DataTest1}
     \end{subfigure}
     \hfill
     \begin{subfigure}[t]{0.5\textwidth}
         \centering
         \includegraphics[width=0.8\textwidth]{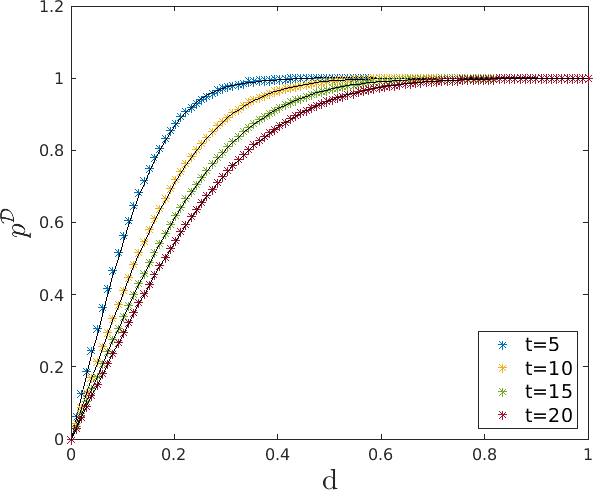}
         \caption{\footnotesize{Considering four different chosen time intants, comparison between the numerical approximation of the pressure $p^\D_h$ (starred lines) and the analytical solution $p^\D$ (black line). }}
         \label{fig:resTest1}
     \end{subfigure}
        \caption{Test 1 -- Terzaghi problem}
        \label{fig:Test1}
\end{figure}
\subsection{Test 2}
In \textit{Test 2}, we consider a single-fractured domain $\D=[0,1]^3$, where the fracture $F$ lies on the plane $z= 0.5$, passing through the barycenter of the domain, and partially ending inside the domain. 
The fracture extends from $0\leq x \leq 1$ and $0\leq y\leq 0.8$.
The domain is subjected to a compression load $\bs{\sigma} \cdot \bs{n}= -1\mathrm{kN/m^2}$ applied on the face $z=1$, and the displacement is allowed along the z-axis on the lateral faces, and the bottom face $z=0$ is clamped.
For the flow problem, a Dirichlet boundary condition $p^\D=1$ is imposed on the face lying on the plane $z=0$, whereas homogeneous Neumann boundary conditions are prescribed on all the remaining faces. Homogeneous Neumann boundary conditions are also enforced along the boundary of the fracture $F$.
A schematic representation of the computational setting is shown in Figure~\ref{fig:Test2_domain}.
In this test, we set the bulk modulus $E=10^3\,\mathrm{kPa}$, Poisson ratio $\nu=0.25$, compressibility constant $\frac{1}{M}=0\,\mathrm{Pa}^{-1}$, hydraulic permeability of the domain $k^{\D} = 10^{-12}\,\mathrm{m}^{2}$ and effective fracture permeability in the tangential plane $k_{\parallel}=10^{-14}\,\mathrm{m}^3$, the effective hydraulic permeability of the fracture in the normal direction $k_{\perp}=10^{-13}\,\mathrm{m}$, the Biot-Willis constant $\alpha = 1$ and a null $g$.
As initial conditions, we prescribe equilibrium displacement and zero pressure over the entire domain.


\begin{figure}
    \centering
    \includegraphics[width=0.8\linewidth]{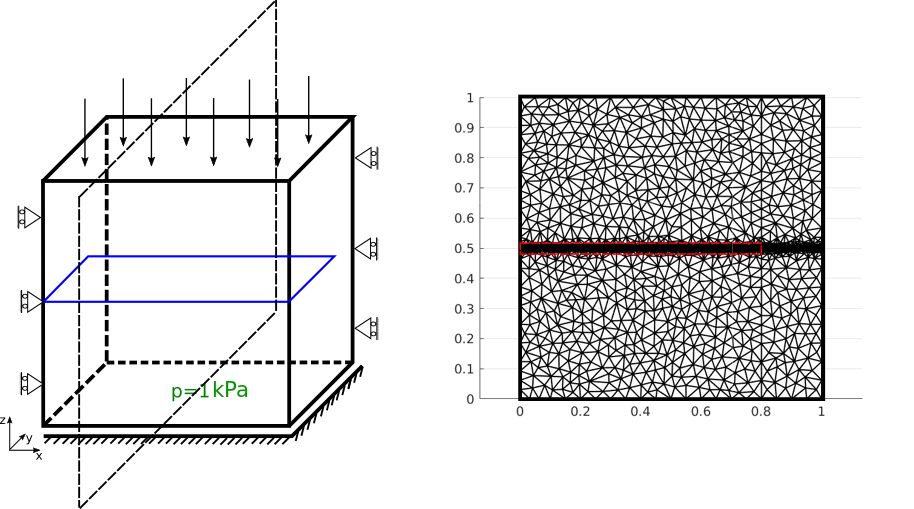}
    \caption{Test 2: mixed-dimensional and analogous 2D equi-dimensional domains }
    \label{fig:Test2_domain}
\end{figure}
\begin{figure}
    \centering
    \begin{subfigure}[t]{0.45\textwidth}
        \includegraphics[width=0.96\linewidth]{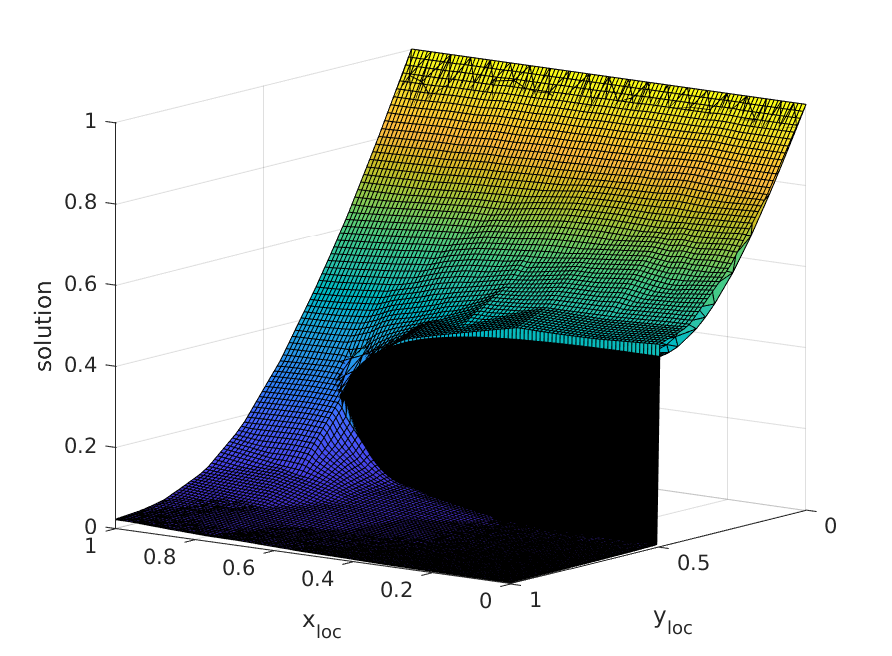}
    \caption{Pressure}
    \label{fig:Test2_pressure_coarse}
    \end{subfigure}
    \hfill
    \begin{subfigure}[t]{0.45\textwidth}
        \includegraphics[width=\linewidth]{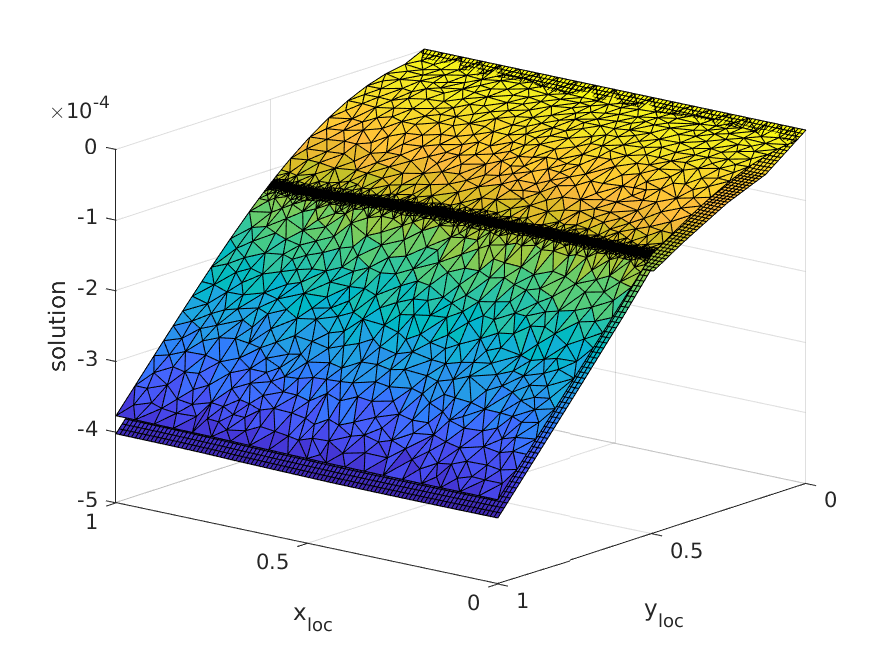}
    \caption{Displacement along $z$-direction}
    \label{fig:Test2_displacement_coarse}
    \end{subfigure}
    \caption{Test 2: comparison of the solution on the coarser grid $h_{c}$, $t=50 \, \mathrm{s}$}
        \label{fig:Test2Comparison_coarse}
\end{figure}
\begin{figure}
    \centering
    \begin{subfigure}[t]{0.45\textwidth}
        \includegraphics[width=0.96\linewidth]{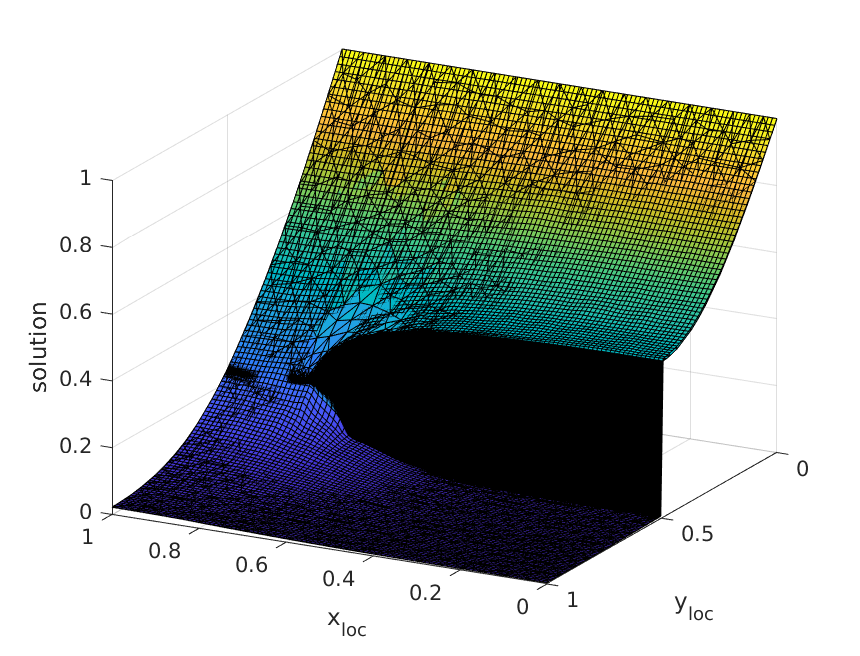}
    \caption{Pressure}
    \label{fig:Test2_pressure_fine}
    \end{subfigure}
    \hfill
    \begin{subfigure}[t]{0.45\textwidth}
        \includegraphics[width=\linewidth]{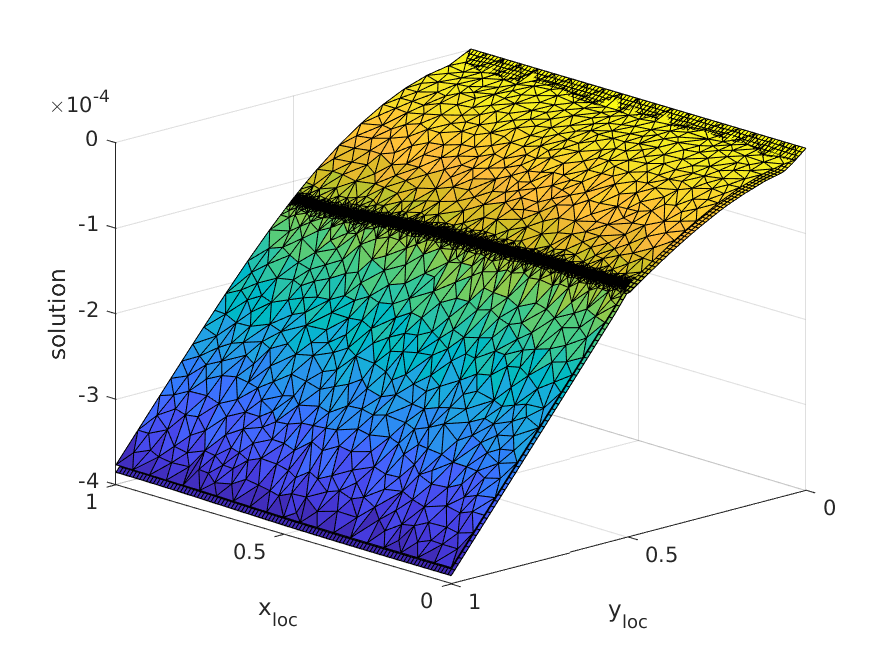}
    \caption{Displacement along $z$-direction}
    \label{fig:Test2_displacement_fine}
    \end{subfigure}
    \caption{Test 2: comparison of the solution on the finer grid $h_{f_2}$, $t=50 \, \mathrm{s}$}
        \label{fig:Test2Comparison_fine}
\end{figure}

This test provides a validation of the proposed approach through the comparison of the solution obtained by the proposed approach with the solution of an analogous equi-dimensional problem. In more details, we solve the 3D mixed-dimensional problem on the domain depicted in Figure~\ref{fig:Test2_domain}, left, with the above data, and we consider the restriction of such solution on the  plane $x=0.5$, that is shown in dashed lines in the same figure.
Then we solve a 2D equi-dimensional problem on the domain shown in Figure~\ref{fig:Test2_domain}, on the right. 
Boundary conditions mimic those of the mixed dimensional domain: for the elasticity problem, the bottom edge of the domain is clamped, only vertical displacement is allowed on the two lateral edges, and a $ -1\mathrm{kN/m}$ force is applied on the top edge; for the pressure problem, a value of $1 \mathrm{kPa}$ is prescribed on the bottom edge and all other edges are insulated. The equi-dimensional fracture is placed in the red-dashed region, with a thickness of $0.01\mathrm{m}$, with centerline on $z=0.5$ and extending between $y=0$ and $y=0.8$. The physical parameters are the same as in the reduced-dimensional problem, we only remark that the hydraulic permeability of the equi-dimensional fracture in the normal direction is now set to $10^{-12}\,\mathrm{m}^{2}$.
  We solve the mixed-dimensional domain on three different meshes: a first coarse mesh is considered, having mesh parameter $h_c=[0.001, \, 0.05,\, 0.05]$, and two finer meshes $h_{f_1}=[0.0001, \, 0.01, \, 0.01]$ and $h_{f_2}=[0.0001, \, 0.005, \, 0.02]$. The mesh for the equi-dimensional problem has maximum element size of $0.001$ outside of the fracture zone, and $10^{-6}$ in the strip $0.495<z<0.505 \ \land 0<y<1$.
For the mixed dimensional problem time step for the Backward Euler scheme is $\Delta t=1 \, \mathrm{s}$, and we use a relative tolerance set to $\mathrm{tol}_{\mathrm{FS}}=10^{-6}$ for the fixed-stress scheme and a tolerance $\mathrm{tol}_{\mathrm{CG}}=10^{-8}$ for CG iterations. The equi-dimensional problem is solved using again Backward Euler with $\Delta t=1$ for time advancing, while we solve monolithically the linear system at each time iteration. We will refer to the solution of the equi-dimensional problem as the \textit{reference} solution in the following.

Figures~\ref{fig:Test2Comparison_coarse}-\ref{fig:Test2Comparison_fine} show a comparison of the solution of the mixed dimensional problem on the $x=0.5$ plane with the solution of the analogous 2D equi-dimensional problem, both at $t=50$. In these figures, the two-dimensional plot of the 3D solutions of the mixed dimensional problem is obtained by evaluating the basis functions of the discrete solution on a Cartesian grid comprising $100\times 100$ nodes. The pressure solution is actually discontinuous across the fracture thanks to the presence of XFEM additional basis functions. In contrast, the pressure solution of the equi-dimensional problem is continuous, but with a very strong gradient across the fracture, which can be correctly resolved thanks to mesh refinement in the fracture region. In more detail, Figure~\ref{fig:Test2_pressure_coarse} displays the pressure solution of the mixed dimensional problem obtained on the coarser mesh $h_c$ superimposed on the reference pressure solution, while Figure~\ref{fig:Test2_displacement_coarse} shows the magnitude of the displacement along the vertical direction obtained on the same coarse mesh along with that of the reference. Figures~\ref{fig:Test2_pressure_fine} and \ref{fig:Test2_displacement_fine} propose the same comparison for pressure and displacement, respectively, but in this case the mixed dimensional solution is obtained on the finer mesh $h_{f_2}$, the solution on the mesh $h_{f_1}$ being, however, indistinguishable from the one proposed, in this kind of plot. A comparison of Figures~\ref{fig:Test2Comparison_coarse}-\ref{fig:Test2Comparison_fine} allows to conclude that, under mesh refinement, the solution obtained approaches the reference one. 

\begin{figure}
    \centering
    \begin{subfigure}[t]{0.32\textwidth}
        \includegraphics[width=\linewidth]{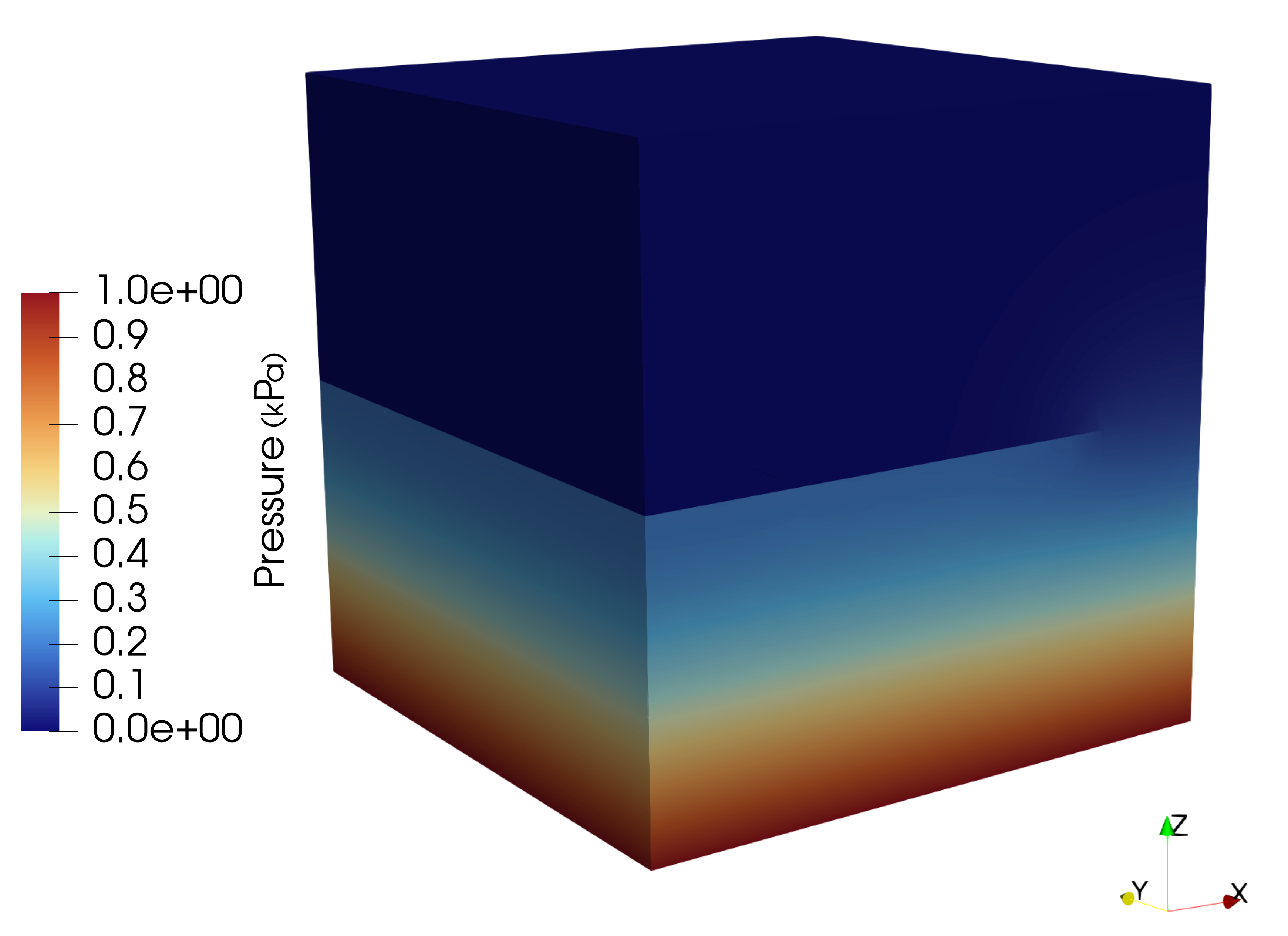}
        \caption{Pressure}
    \label{fig:Test2_Sol3Dt25Pressure}
    \end{subfigure}
    \begin{subfigure}[t]{0.32\textwidth}
    \centering
        \includegraphics[width=0.8\linewidth]{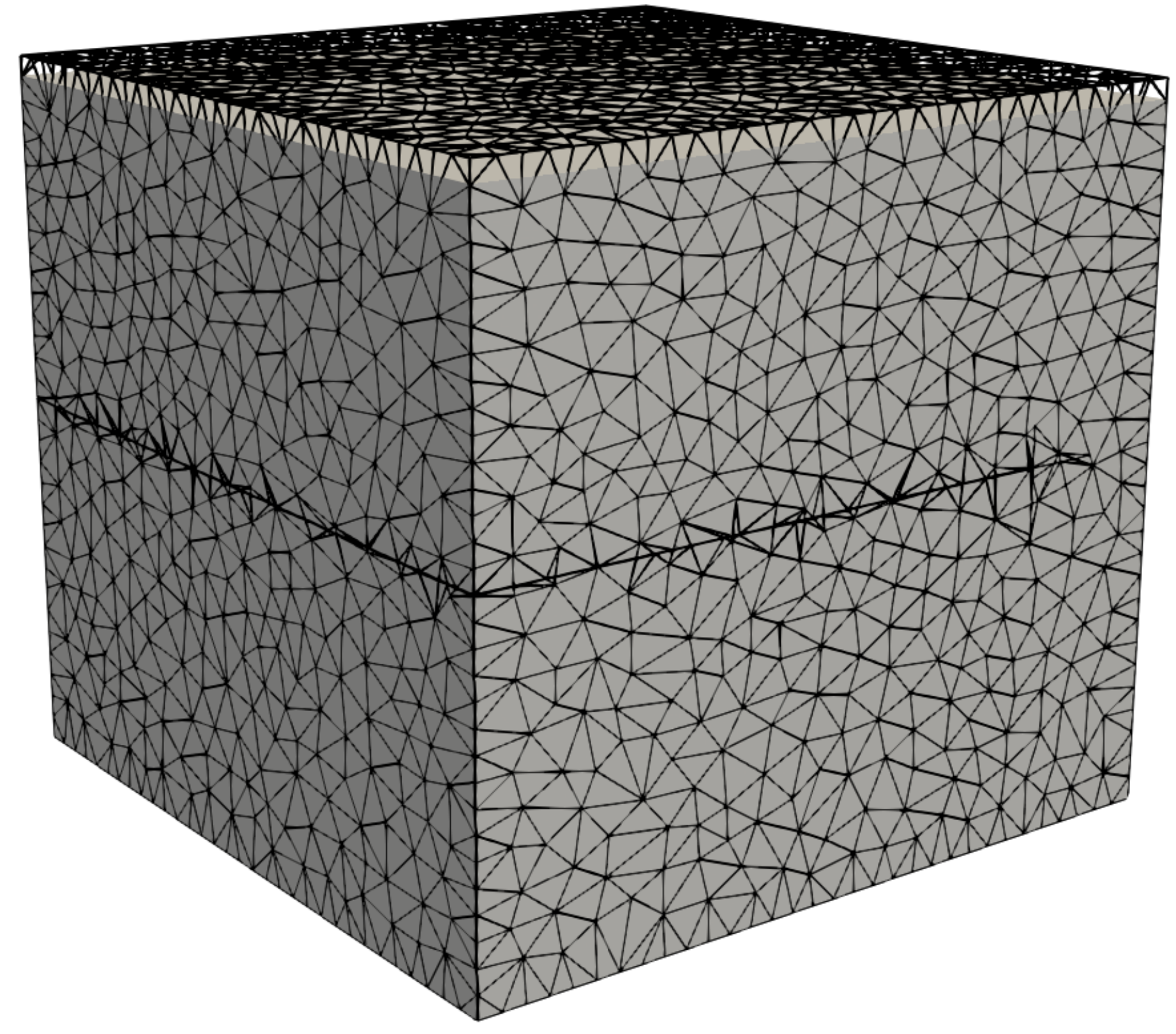}
        \caption{Displacement (200x)}
    \label{fig:Test2_Sol3Dt25Displacement}
    \end{subfigure}
    \begin{subfigure}[t]{0.32\textwidth}
        \includegraphics[width=0.9\linewidth]{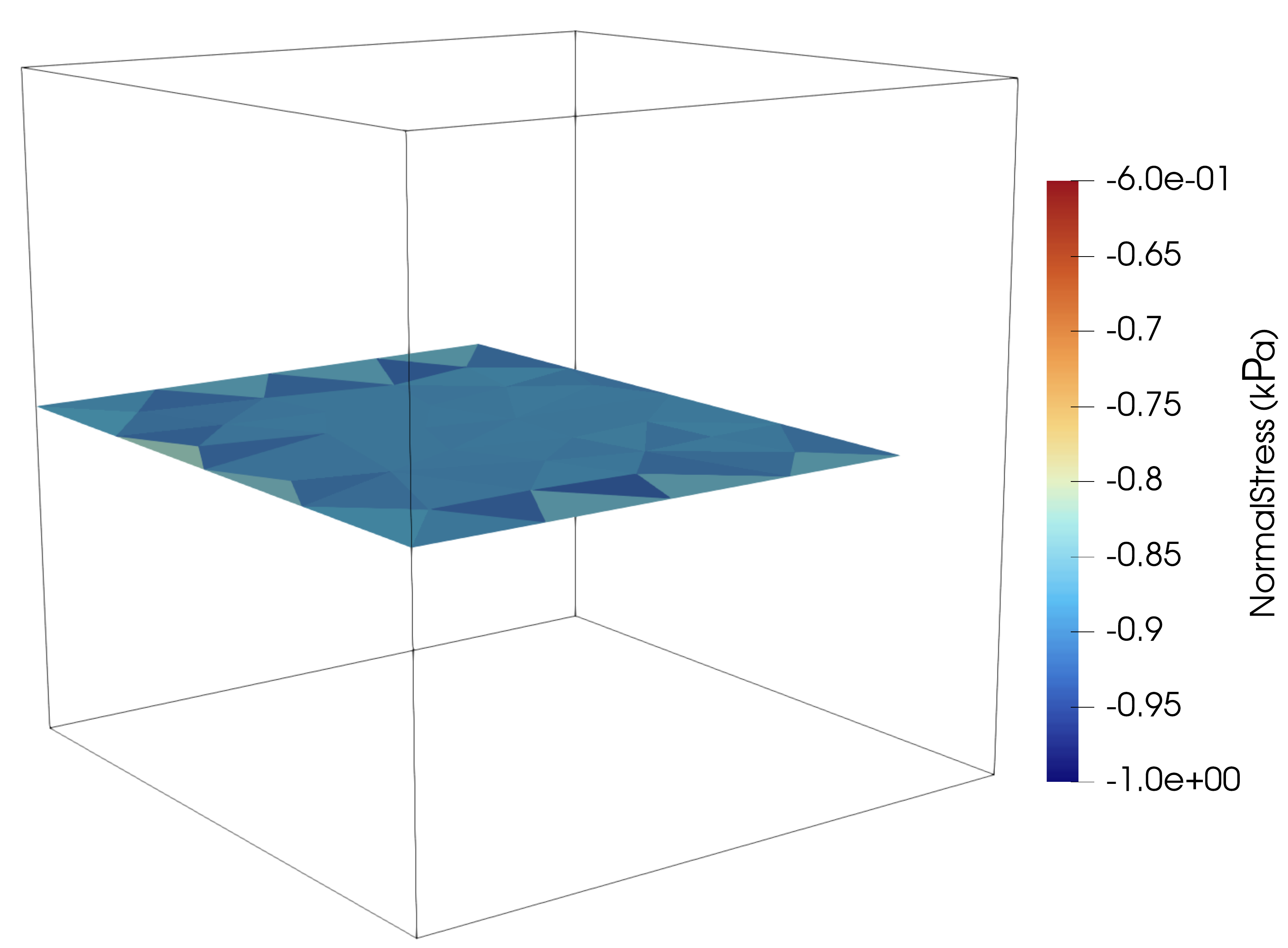}
        \caption{Normal stress on fracture}
    \label{fig:Test2_Sol3Dt25Fracture}
    \end{subfigure}
    \caption{Test 2: Solution at $t=25 \, \mathrm{s}$ on mesh $h_{f_2}$}
        \label{fig:Test2_3Dt25}
\end{figure}
\begin{figure}
    \centering
    \begin{subfigure}[t]{0.32\textwidth}
        \includegraphics[width=\linewidth]{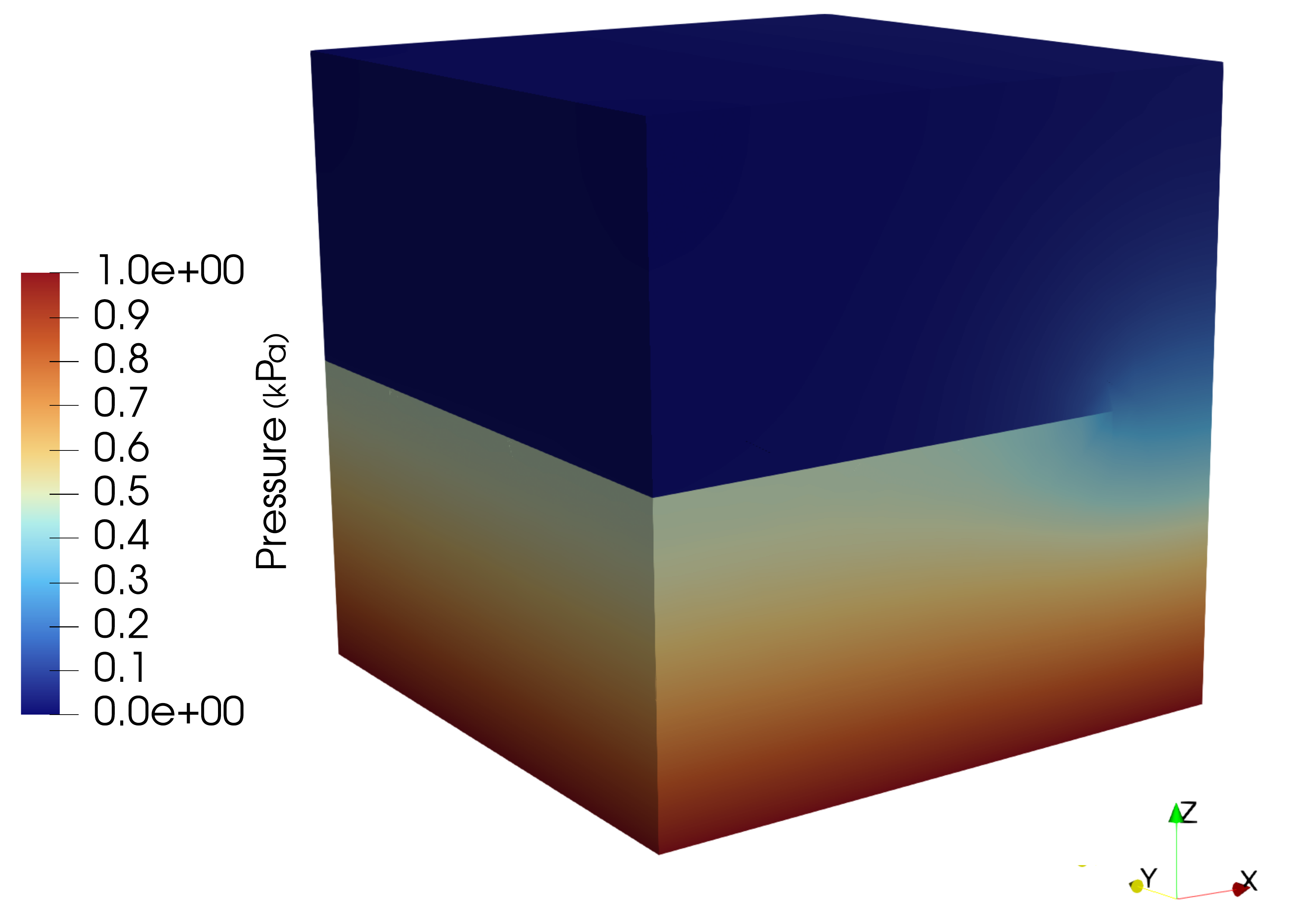}
        \caption{Pressure}
    \label{fig:Test2_Sol3Dt50Pressure}
    \end{subfigure}
    \begin{subfigure}[t]{0.32\textwidth}
    \centering
        \includegraphics[width=0.8\linewidth]{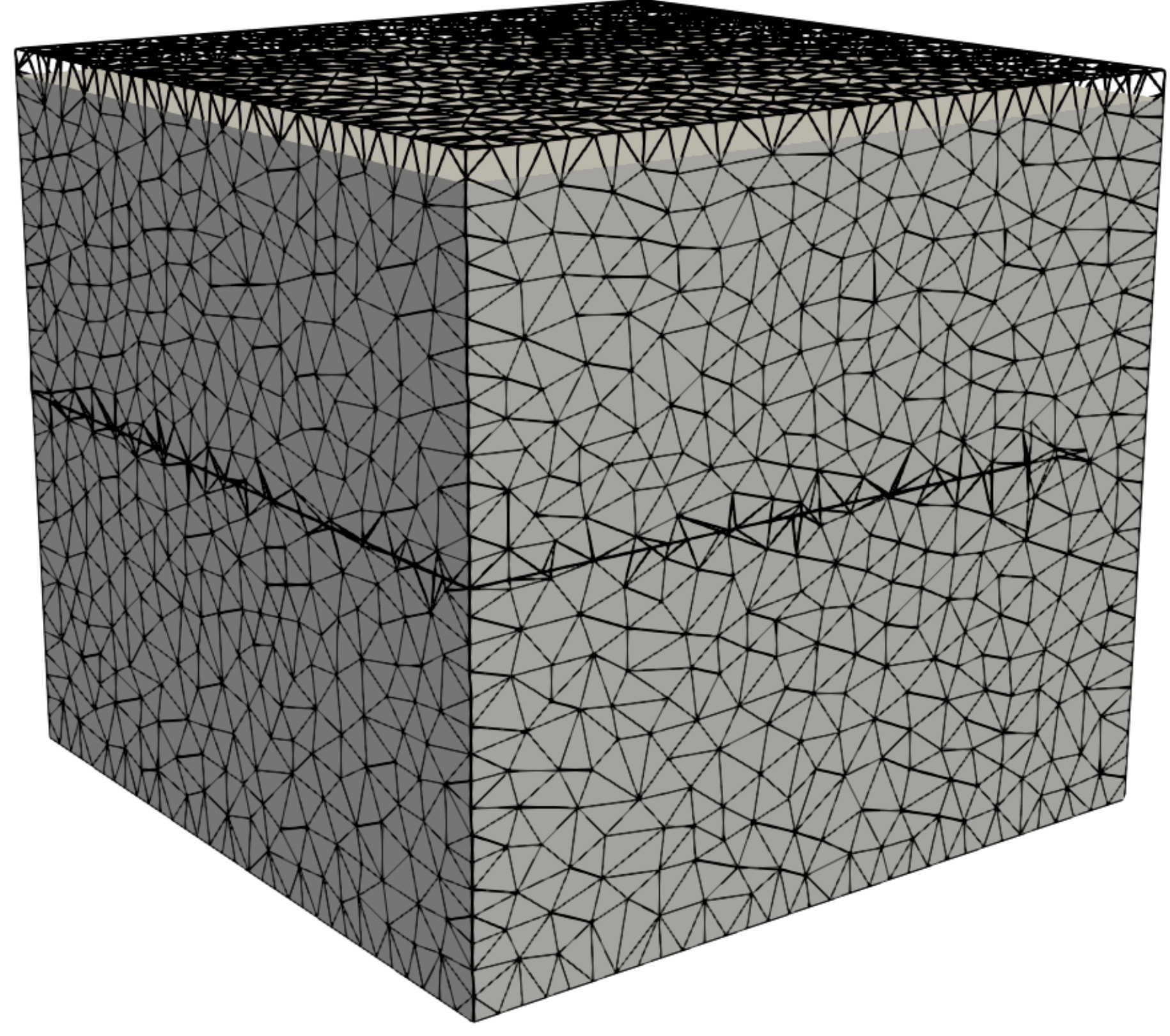}
        \caption{Displacement (200x)}
    \label{fig:Test2_Sol3Dt50Displacement}
    \end{subfigure}
    \begin{subfigure}[t]{0.32\textwidth}
        \includegraphics[width=0.9\linewidth]{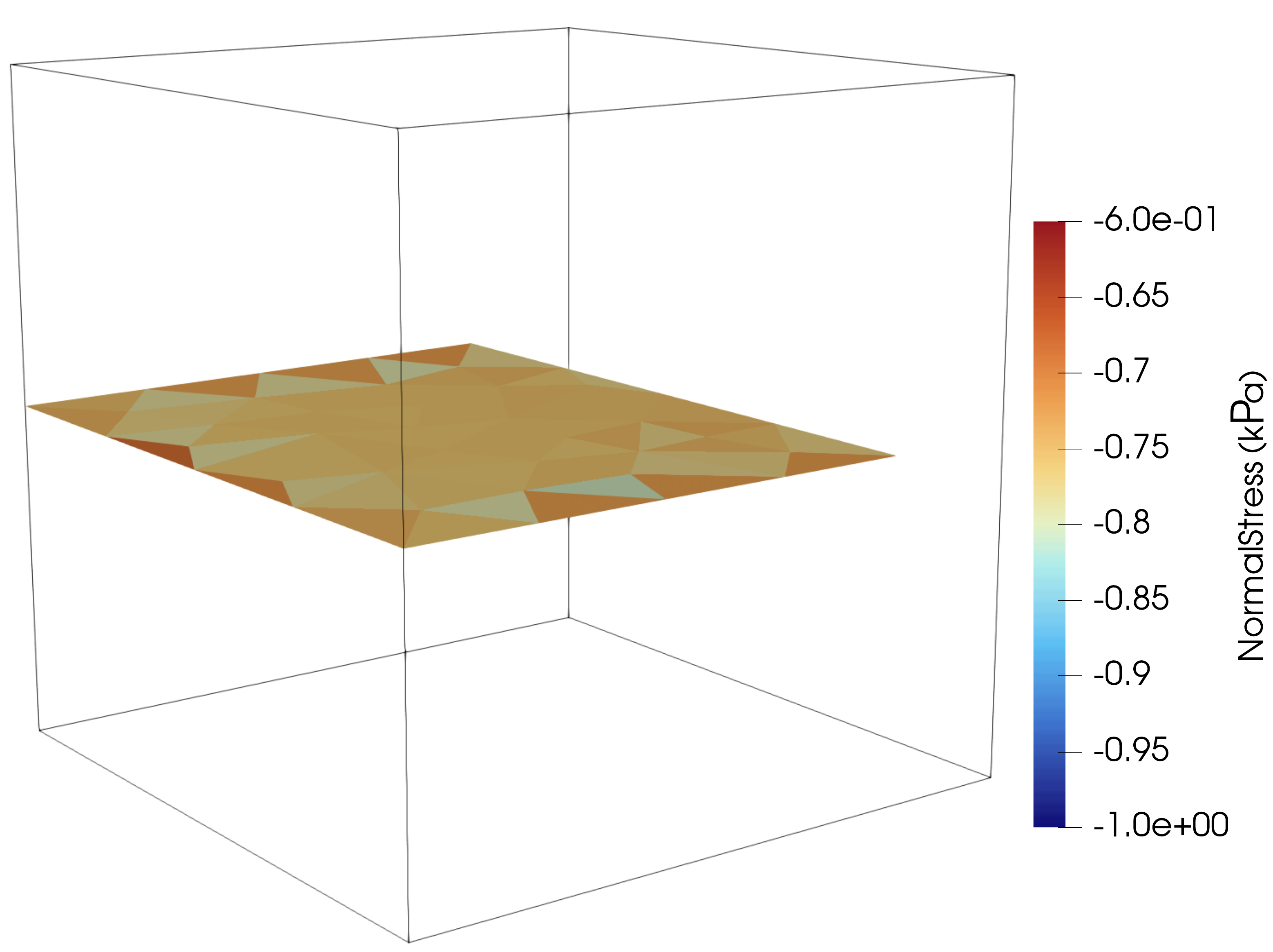}
        \caption{Normal stress on fracture}
    \label{fig:Test2_Sol3Dt50Fracture}
    \end{subfigure}
    \caption{Test 2: Solution at $t=50 \, \mathrm{s}$ on mesh $h_{f_2}$}
        \label{fig:Test2_3Dt50}
\end{figure}

Figures~\ref{fig:Test2_3Dt25}-\ref{fig:Test2_3Dt50} reproduce the 3D plot of the mixed dimensional solution at time $t=25$ and $t=50$, respectively. In both images, the leftmost picture represents the pressure solution, the middle one the displacement, and the rightmost the computed normal stress on the fracture. The discontinuity of the pressure across the fracture can be clearly seen in Figures~\ref{fig:Test2_Sol3Dt25Pressure} and \ref{fig:Test2_Sol3Dt50Pressure}, being instead continuous away from the fracture. It should be noted that the computational mesh is nonconforming to the fracture. The mesh, visible in Figures~\ref{fig:Test2_Sol3Dt25Displacement} and \ref{fig:Test2_Sol3Dt50Displacement} is, indeed only for graphical representation purposes, and is obtained from the computational mesh cutting the original tetrahedra into sub-tetrahedrons that do not cross the fracture. The discrete solution is finally evaluated on the newly generated cells and displayed on such unstructured grid. The discontinuous basis functions of the discrete solution provide two different values for the solution across the interface. The pressure solution can thus exhibit a jump across the fracture, that grows with time, as expected.  The displacement, in Figures~\ref{fig:Test2_Sol3Dt25Displacement} and \ref{fig:Test2_Sol3Dt50Displacement}, is shown by moving the mesh according to the computed displacement solution amplified by a factor $200$. The domain shape at $t=0$ is also shown in the same pictures as a white solid, to better highlight the deformation occurring at the subsequent times. Displacements are weakly continuous thanks to the constraint \eqref{discreteWeakCont}, and discontinuous basis functions are present also in the discrete functional space for this problem. Comparing solutions at $t=25$ and $t=50$ it can be seen that, as expected the displacement in the vertical direction dominates. Also, it is noted that the presence of the fracture does not break the symmetry in the vertical displacement, as also happens for the reference solution. The continuity condition \eqref{pb:contVar:EqLambda} on the displacement allows the computation of the  stresses on the fracture. The normal component of such stress, i.e. $\bs{\Lambda}\cdot\bs{n}$, is displayed in Figures~\ref{fig:Test2_Sol3Dt25Fracture} and \ref{fig:Test2_Sol3Dt50Fracture} for the two considered time-frames. As expected, the distribution of the normal stress is nearly constant on the fracture, and reduces with time. 

\begin{figure}
    \centering
    \begin{subfigure}[t]{0.45\textwidth}
        \includegraphics[width=0.96\linewidth]{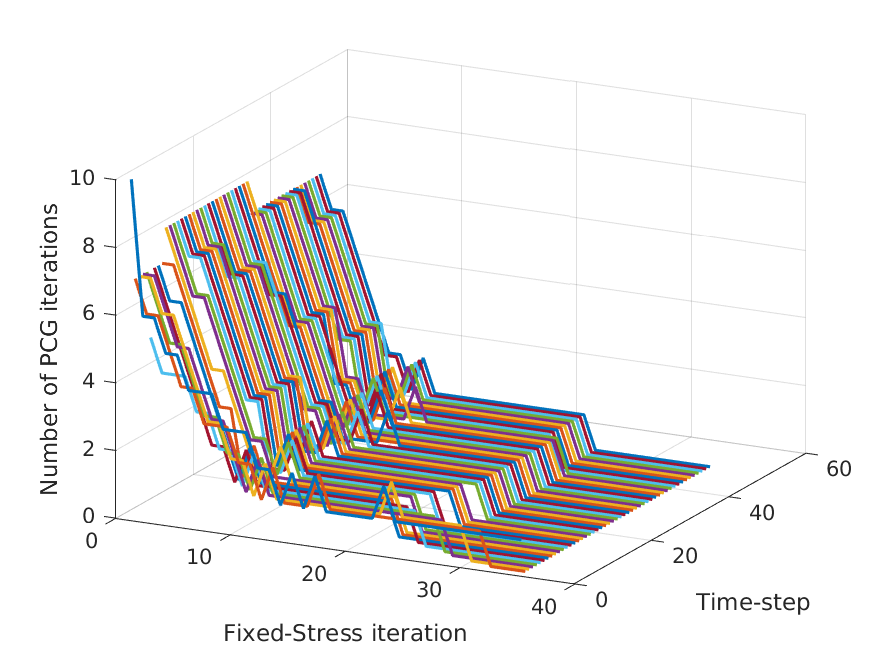}
    \caption{Mesh $h_c$}
    \label{fig:IterInfo_coarse}
    \end{subfigure}
    \hfill
    \begin{subfigure}[t]{0.45\textwidth}
        \includegraphics[width=0.96\linewidth]{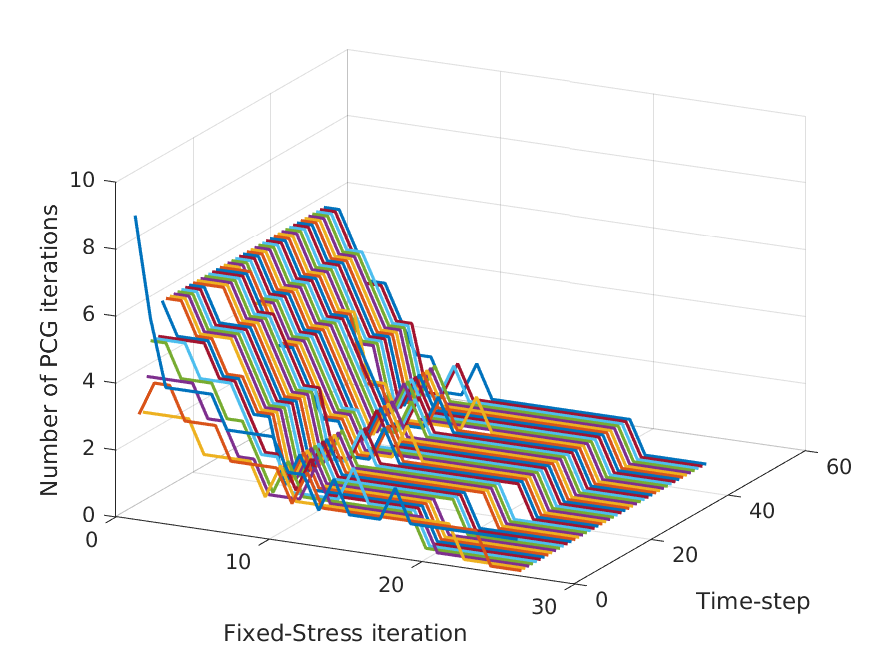}
    \caption{Mesh $h_{f_2}$}
    \label{fig:IterInfo_fine2}
    \end{subfigure}
    \caption{Test 2: Iterations of the fixed-stress and CG schemes at each time-step for selected meshes}
        \label{fig:Test2IterationInfo}
\end{figure}

Figure~\ref{fig:Test2IterationInfo} reports a plot of the iterations of the fixed stress method and of the CG method required to reach the prescribed tolerances at each time step, i.e. $\mathrm{tol}_{\mathrm{FS}}=10^{-6}$ and $\mathrm{tol}_{\mathrm{CG}}=10^{-8}$. 
Figure~\ref{fig:IterInfo_coarse} relates to the computation of the solution on the coarse mesh $h_c$ and Figure~\ref{fig:IterInfo_fine2} to the finer mesh $h_{f_2}$. For both meshes, it can be seen that the number of iterations required by the fixed stress scheme remains almost unchanged for all time steps, with small variations of a few iterations. The CG iteration counts show similar trends across time steps: a larger number of iterations is needed for the first fixed stress loops, but after about 10 fixed stress loops, the number of CG iterations drops to about 2, and, finally, just 1 CG iteration is sufficient to meet the required stopping criterion for the last fixed stress loops. This makes the proposed strategy very promising for a computationallyly efficient resolution of this kind of problems. The number of fixed stress iterations is higher for the coarser mesh with a value of about 35 iterations, which reduces to about 24 iterations on mesh $h_{f_2}$. The behavior observed for mesh $h_{f_2}$ is also representative of the trend obtained for mesh $h_{f_1}$.

\subsection{Test 3}
The third numerical example is set in the same cubic domain as in the previous test, with four embedded planar fractures, as shown in Figure~\ref{fig:Test3_domain}. 
As in the previous test, the material bulk modulus is $E=10^3\,\mathrm{kPa}$, the Poisson ratio $\nu=0.25$, and the compressibility constant $\frac{1}{M}=0\,\mathrm{Pa}^{-1}$. Hydraulic permeability is $k^\D = 10^{-12}\,\mathrm{m}^{2}$ for the bulk domain, the effective fracture permeability in the tangential plane is $k_{\parallel}=10^{-14}\,\mathrm{m}^3$, for all fractures, while the effective fracture permeability across the plane is $k^1_{\perp}=10^{-13}\,\mathrm{m}$, for $F_1$, $k^{2,3}_{\perp}=10^{-12}\,\mathrm{m}$ for fractures $F_2$ and $F_3$ and $k^4_{\perp}=10^{-10}\,\mathrm{m}$ for $F_4$. The Biot-Willis constant is $\alpha = 1$ and $g=0$.
Concerning the displacement, for this test, only the bottom face, edges and nodes of the 3D domain are clamped, all other faces having Neumann boundary conditions instead, that are null on the sides. A $-1\mathrm{kN/m^2}$ vertical stress is applied, instead, on the top face. For the pressure problem, homogeneous Neumann boundary conditions are enforced on the whole boundary of the 3D domain and of the fractures. 
A volumetric source term for the pressure is imposed, equal to $10^{-3}\mathrm{s}^{-1}$ in a sphere of radius $0.1 \, \, \mathrm{m}$ centered in $[0.75, \, 0.8, \, 0.2]\,\,\mathrm{m}$, and zero elsewhere.   
As initial conditions, we prescribe equilibrium displacement and zero pressure over the entire domain.

The computational mesh is shown in Figure~\ref{fig:Test3_mesh}, with mesh parameter $h=[0.0001, \, 0.02,\, 0.02]$. It can be seen that the 3D cells arbitrarily cross the fracture planes.

\begin{figure}
    \centering
    \begin{subfigure}[t]{0.45\textwidth}
    \centering
        \includegraphics[width=0.75\linewidth]{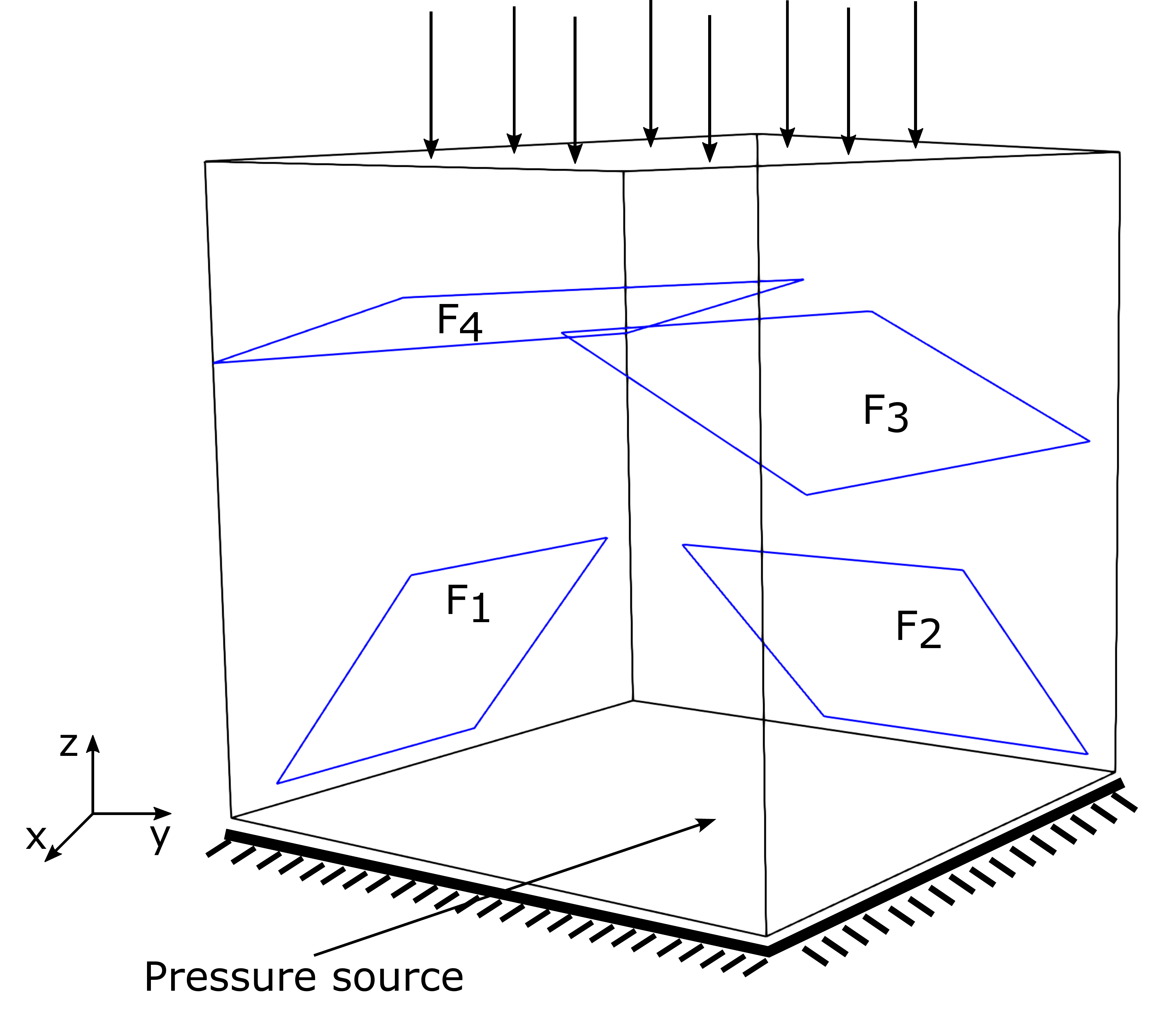}
    \caption{Domain details}
    \label{fig:Test3_domain}
    \end{subfigure}
    \hfill
    \begin{subfigure}[t]{0.45\textwidth}
        \includegraphics[width=\linewidth]{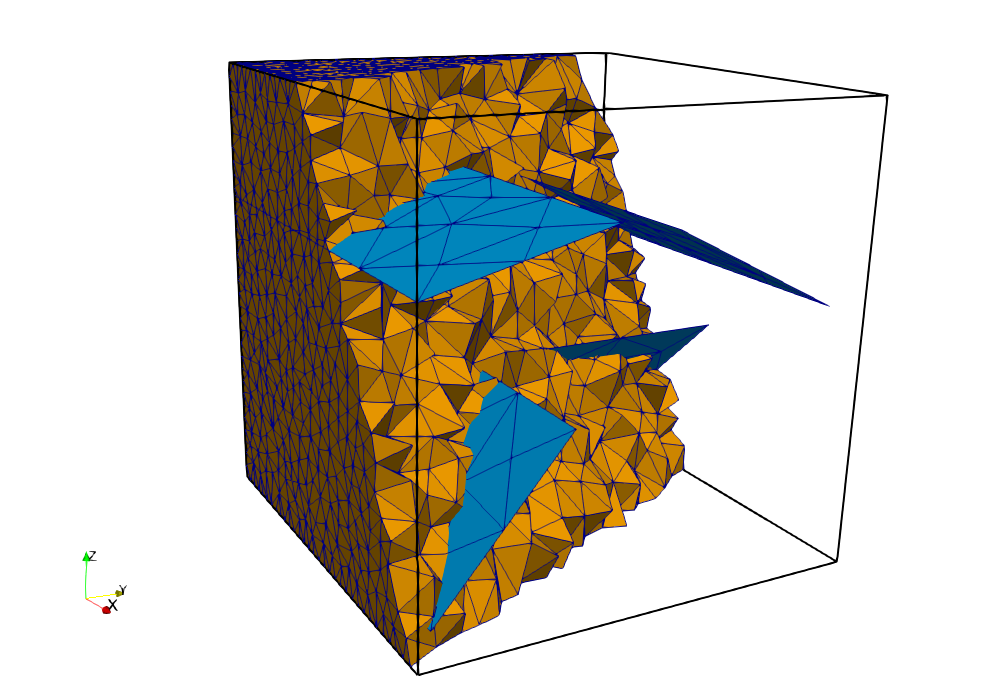}
    \caption{3D and 2D computational mesh}
    \label{fig:Test3_mesh}
    \end{subfigure}
    \caption{Test 3: Computational domain and mesh detail}
        \label{fig:Test3_domain_and_mesh}
\end{figure}

\begin{figure}
    \centering
    \begin{subfigure}[t]{0.32\textwidth}
        \includegraphics[width=\linewidth]{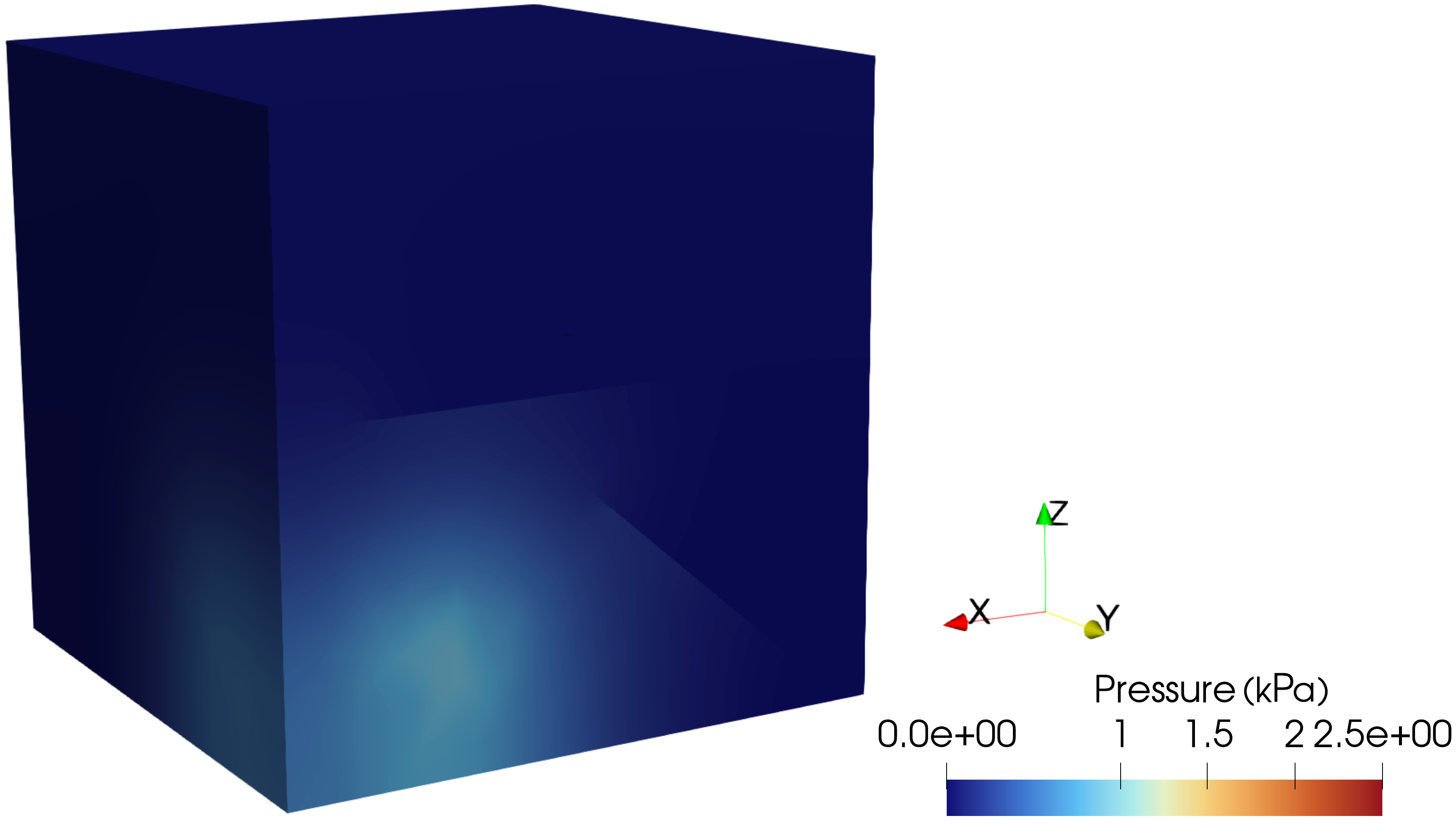}
        \caption{Pressure}
    \label{fig:Test3_Sol3Dt25Pressure}
    \end{subfigure}
    \begin{subfigure}[t]{0.32\textwidth}
        \includegraphics[width=0.75\linewidth]{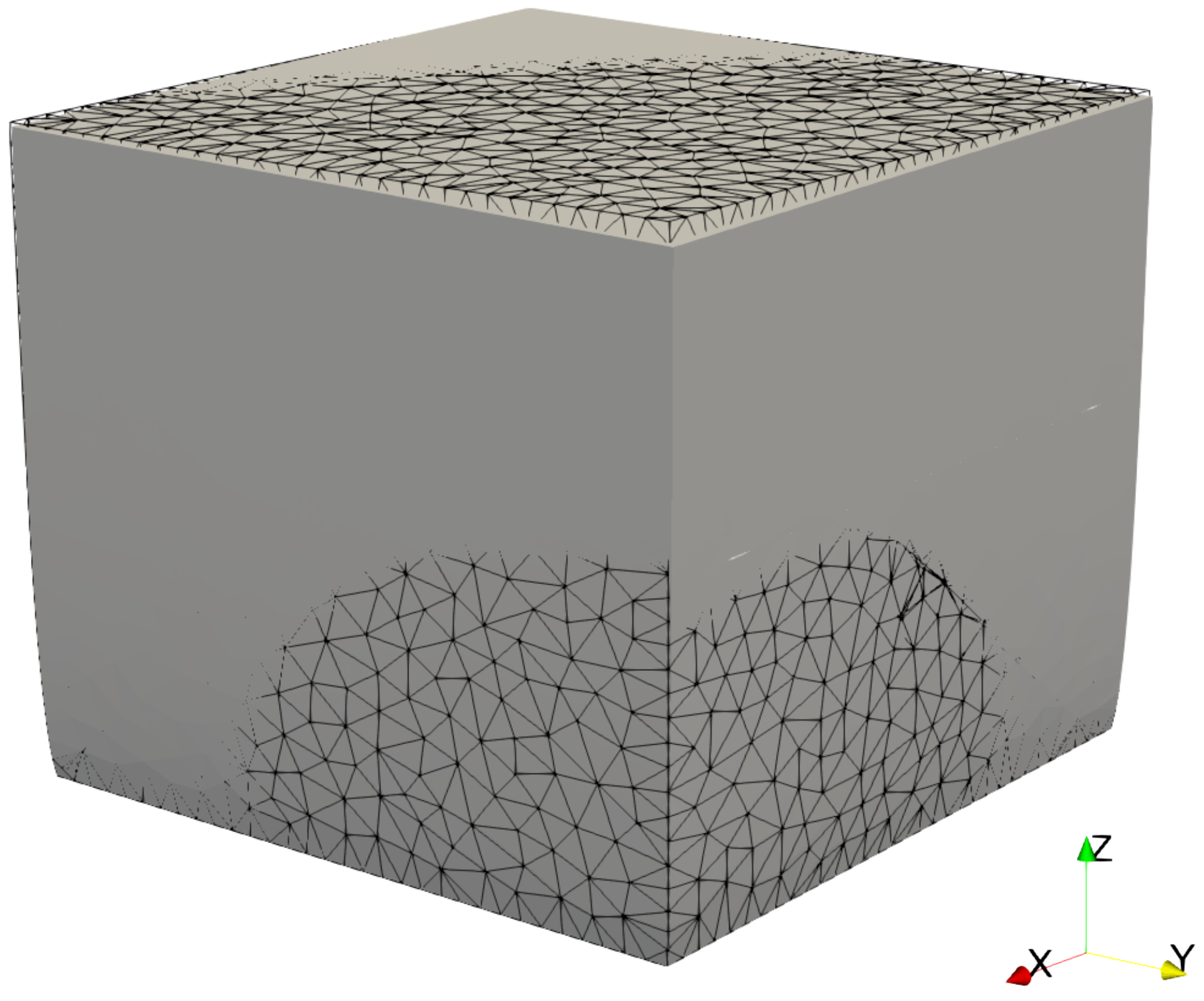}
        \caption{Displacement (200x)}
    \label{fig:Test3_Sol3Dt25Displacement}
    \end{subfigure}
    \begin{subfigure}[t]{0.32\textwidth}
        \includegraphics[width=\linewidth]{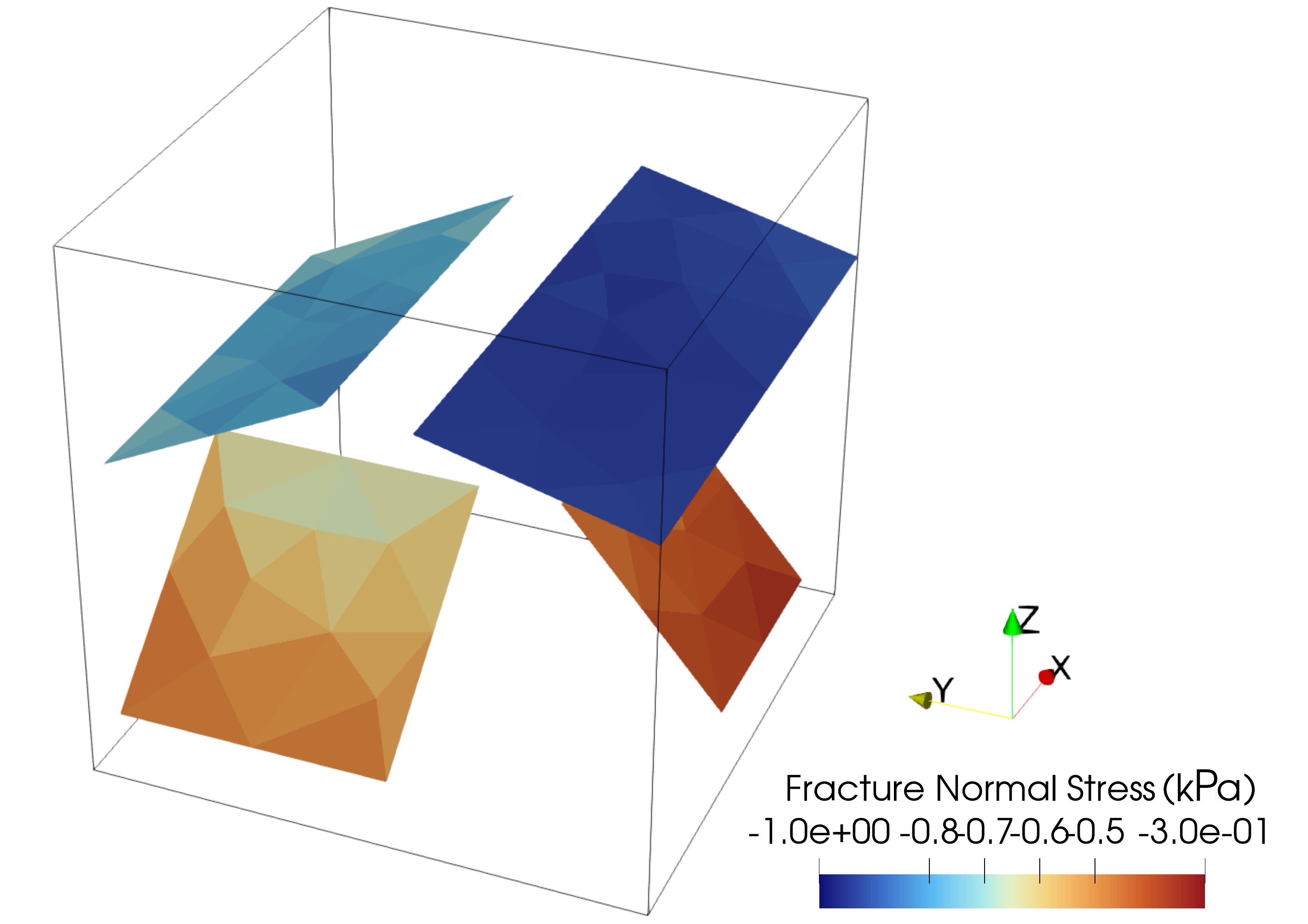}
        \caption{Normal stress on fracture}
    \label{fig:Test3_Sol3Dt25Fracture}
    \end{subfigure}
    \caption{Test 3: Solution at $t=25 \, \mathrm{s}$}
        \label{fig:Test3_3Dt25}
\end{figure}

\begin{figure}
    \centering
    \begin{subfigure}[t]{0.32\textwidth}
        \includegraphics[width=\linewidth]{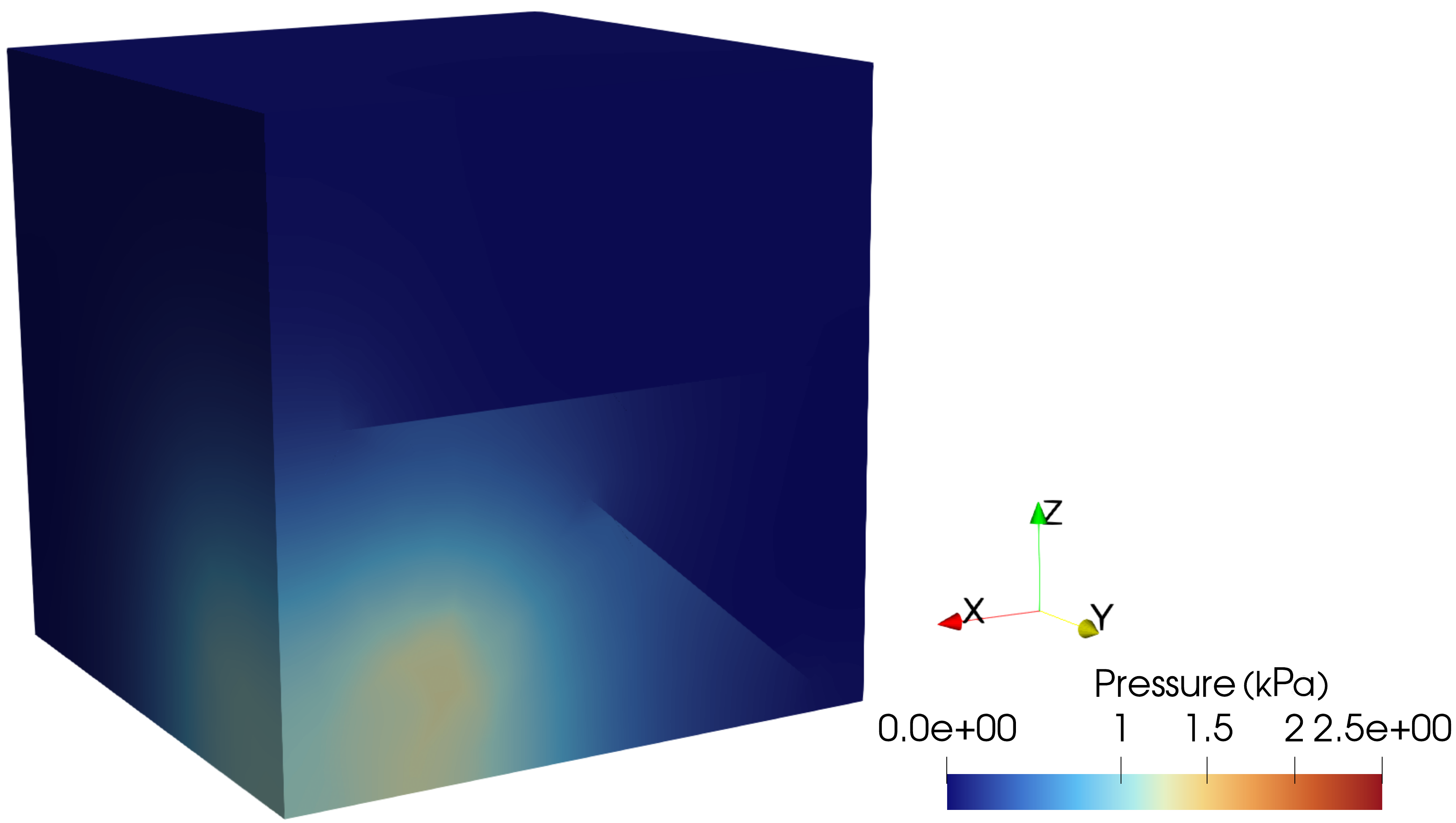}
        \caption{Pressure}
    \label{fig:Test3_Sol3Dt50Pressure}
    \end{subfigure}
    \begin{subfigure}[t]{0.32\textwidth}
        \includegraphics[width=0.75\linewidth]{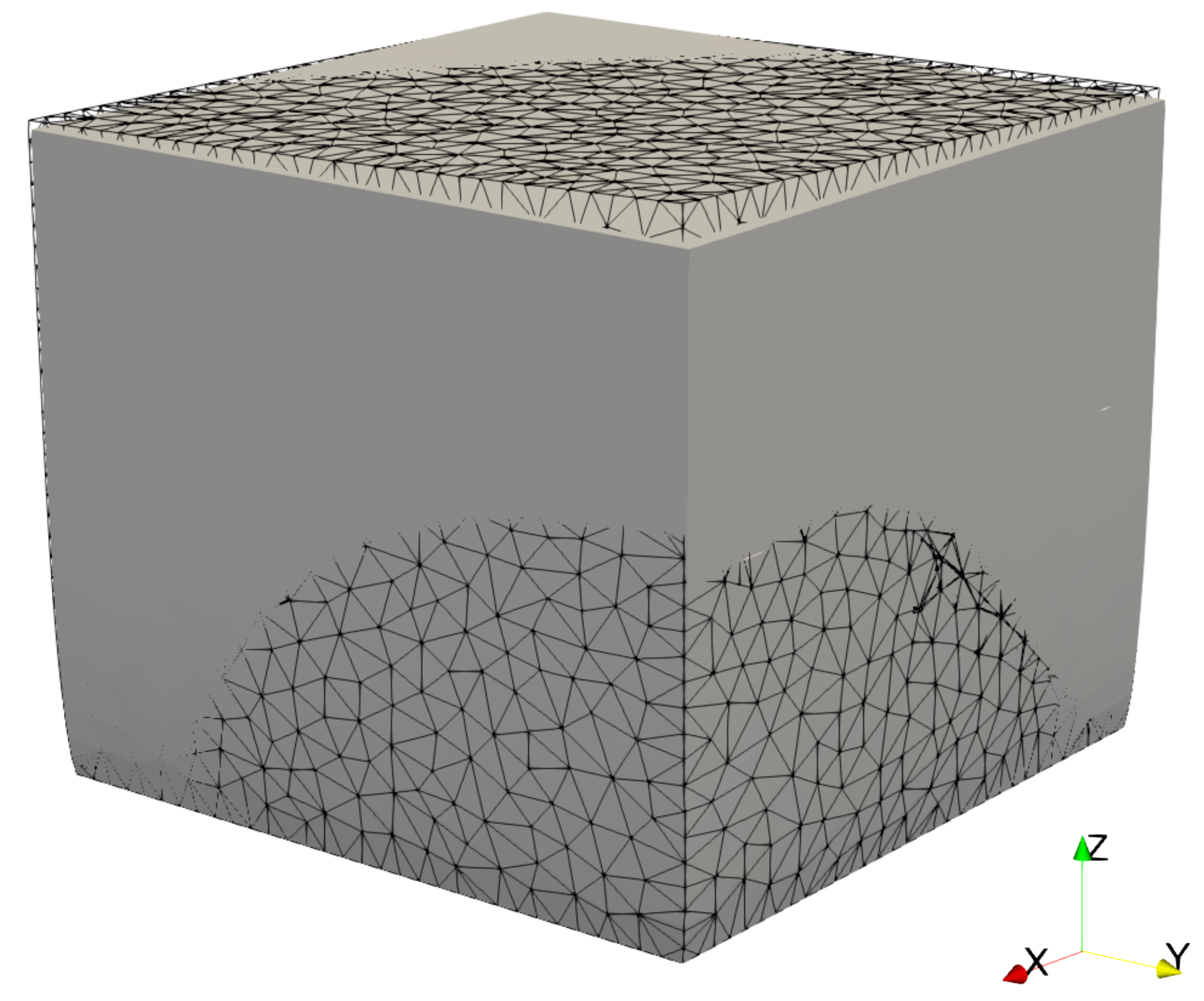}
        \caption{Displacement (200x)}
    \label{fig:Test3_Sol3Dt50Displacement}
    \end{subfigure}
    \begin{subfigure}[t]{0.32\textwidth}
        \includegraphics[width=\linewidth]{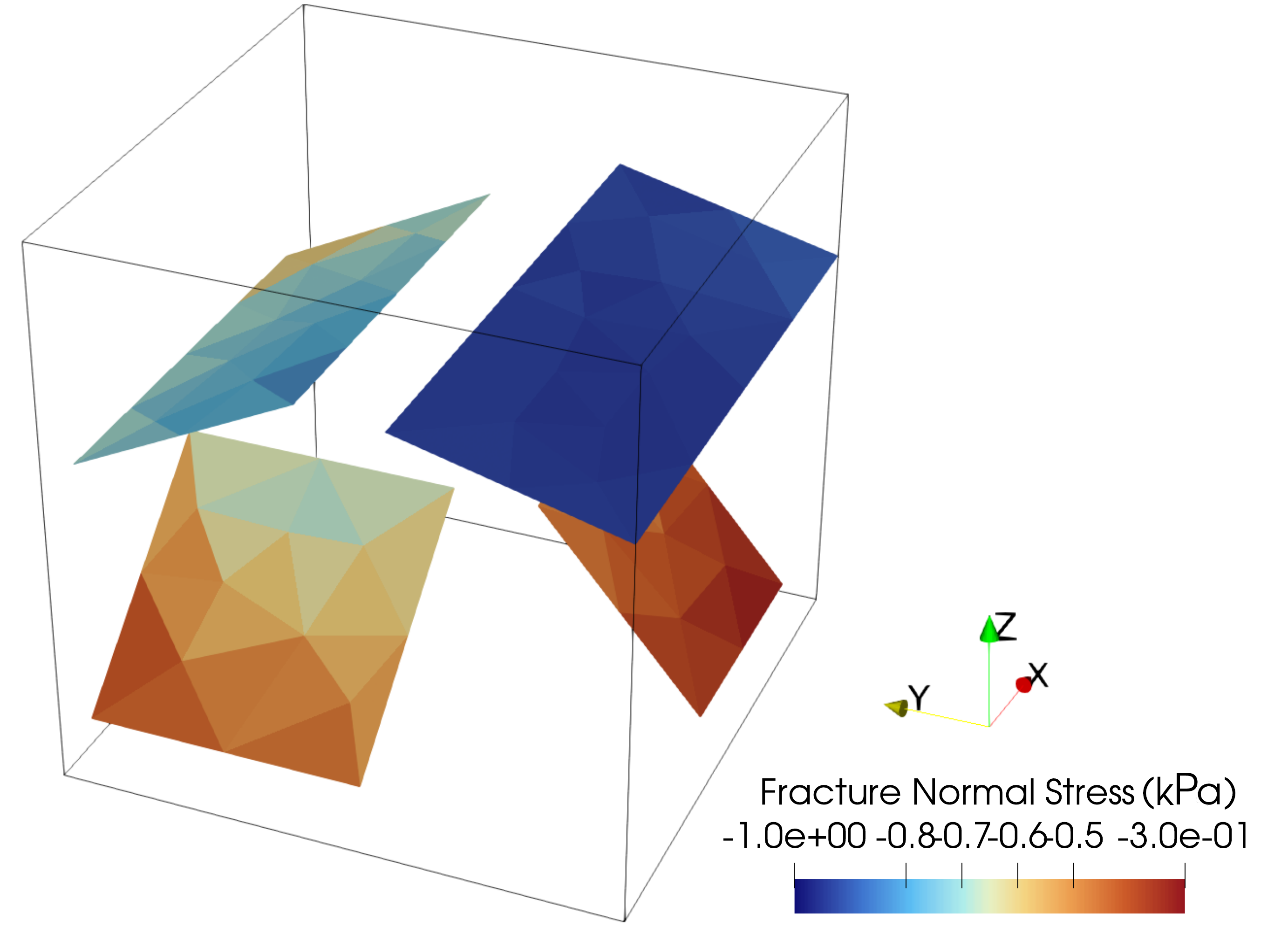}
        \caption{Normal stress on fracture}
    \label{fig:Test3_Sol3Dt50Fracture}
    \end{subfigure}
    \caption{Test 3: Solution at $t=50 \, \mathrm{s}$}
        \label{fig:Test3_3Dt50}
\end{figure}

\begin{figure}
    \centering
    \begin{subfigure}[t]{0.49\textwidth}
        \includegraphics[width=0.75\linewidth]{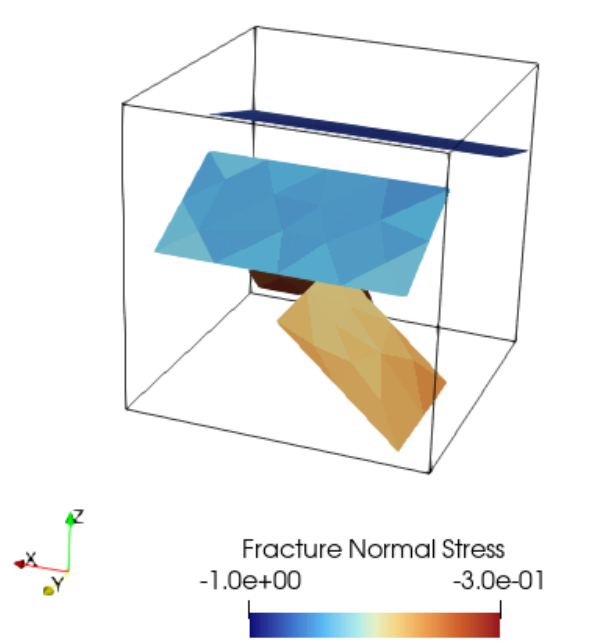}
        \caption{Normal stress solution at $t=0 \, \mathrm{s}$}
    \label{fig:Test3_NormalStress_t0}
    \end{subfigure}
    \begin{subfigure}[t]{0.49\textwidth}
        \includegraphics[width=0.75\linewidth]{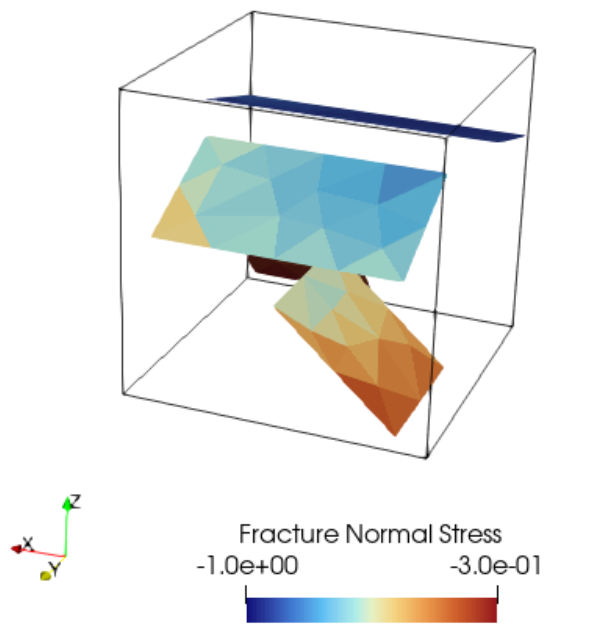}
        \caption{Normal stress solution at $t=50 \, \mathrm{s}$}
    \label{fig:Test3_NormalStress_t50}
    \end{subfigure}
    \caption{Test 3: Detail of normal stress $\bs{\Lambda}^j\cdot \bs{n}^j$ solution on fracture $F_j$, $j=2,3$.}
        \label{fig:Test3_NormalStress}
\end{figure}

\begin{figure}
    \centering
    \includegraphics[width=0.5\linewidth]{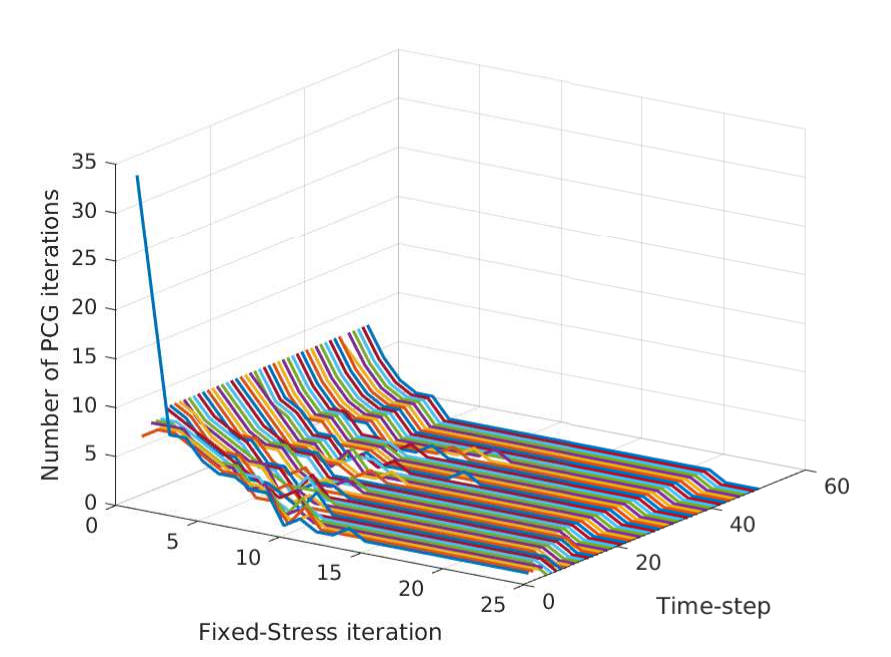}
    \caption{Test 3: Fixed stress and CG iteration count for each time step}
    \label{fig:Test3_iterationInfo}
\end{figure}

Figures~\ref{fig:Test3_3Dt25}-\ref{fig:Test3_3Dt50} show the solution at two different snapshots, at $t=25 \mathrm{s}$ and $t=50 \mathrm{s}$. Panels~\ref{fig:Test3_Sol3Dt25Pressure} and \ref{fig:Test3_Sol3Dt50Pressure} report the computed pressure. It can be noticed that the solution is affected by the presence of the low-permeable fractures in the normal direction. This has an impact also on the deformation, shown in pictures~\ref{fig:Test3_Sol3Dt25Displacement} and \ref{fig:Test3_Sol3Dt50Displacement}. The displacement is represented by deforming the elements of the mesh according to the computed solution, amplified 200 times. The displacement is clearly more pronounced where the higher values of the pressure are observed. Finally, panels~\ref{fig:Test3_Sol3Dt25Fracture} and \ref{fig:Test3_Sol3Dt50Fracture} describe the normal stress on the fracture planes. It is observed that the solution on fractures $F_1$ and $F_4$ remains almost unchanged at the two considered times, whereas larger variations are seen on the other two fractures. The normal stress on fractures $F_2$ and $F_3$ is thus shown in Figure~\ref{fig:Test3_NormalStress}, with a different orientation to better highlight the solution, at $t=0 \, \mathrm{s}$ and at $t=50 \, \mathrm{s}$. It can be seen that the absolute value of the normal stress progressively decreases on fracture $F_3$, the reduction being stronger towards the boundary closer to the pressure source. On fracture $F_2$, instead, the absolute value of the normal stress increases with respect to the initial value closer to the top edge of the fracture, whereas it slightly reduces closer to the bottom. 
Figure~\ref{fig:Test3_iterationInfo} reports the fixed-stress and CG iteration counts required at each time step to reach the desired stopping criteria, set to $\mathrm{tol}_{\mathrm{FS}}=10^{-7}$ and $\mathrm{tol}_{\mathrm{CG}}=10^{-8}$. A trend similar to the one of the previous test is observed. The number of fixed-stress iterations remains stable across the different time-steps. The number of CG iterations, instead rapidly drops from $5-10$ iterations for the first $5-10$ fixed-stress iterations to only $1-2$ CG iterations for the remaining fixed-stress loops. This behavior is similar for all time-steps. 
\section{Conclusions}
\label{sec:Conclusions}

The solution of coupled flow and mechanics problems in mixed-dimensional domains is computationally demanding due to the large size and geometric complexity of the computational domain. The proposed approach addresses these challenges by combining an operator-splitting technique, namely the fixed-stress scheme, with a domain decomposition strategy specifically designed for mixed-dimensional problems on non-conforming meshes. The domain decomposition framework enables the bulk domain and the lower-dimensional inclusions to be solved independently, while preserving their coupling. The use of non-conforming meshes greatly simplifies mesh generation, as the mesh resolution can be selected independently of the geometric complexity of the fracture network. The numerical results demonstrate the effectiveness and the viability of the proposed methodology, and include a qualitative study of its accuracy through the comparison with a standard FEM based solution. In particular, the combination of the fixed-stress scheme with the conjugate gradient solver for the flow problem proves to be promising in terms of efficiency. The proposed strategy thus appears to be a viable option for the simulation of large-scale problems in fractured porous media.   
\printbibliography

@Article{BPSa,
  AUTHOR =       {S. Berrone and S. Pieraccini and S. Scial{\`o}},
  TITLE =        {A {PDE}-constrained optimization formulation for discrete
                  fracture network flows},
  JOURNAL =      {SIAM J. Sci. Comput.},
  FJOURNAL =     {SIAM Journal on Scientific Computing},
  VOLUME =       35,
  YEAR =         2013,
  NUMBER =       2,
  PAGES =        {B487--B510},
  ISSN =         {1064-8275},
  CODEN =        {SJOCE3},
  MRCLASS =      {65N30 (49M25 65K10 65N15 65N50 76S05)},
  MRNUMBER =     3038028,
  MRREVIEWER =   {Marius Ghergu},
  DOI =          {10.1137/120865884},
}

@article{Berge2020,
author = {Berge, R. L. and Berre, I. and Keilegavlen, E. and Nordbotten, J. M. and Wohlmuth, B.},
title = {Finite volume discretization for poroelastic media with fractures modeled by contact mechanics},
journal = {International Journal for Numerical Methods in Engineering},
volume = {121},
number = {4},
pages = {644-663},
doi = {10.1002/nme.6238},
year = {2020}
}

@article{Both2017,
title = {Robust fixed stress splitting for Biot’s equations in heterogeneous media},
journal = {Applied Mathematics Letters},
volume = {68},
pages = {101-108},
year = {2017},
issn = {0893-9659},
doi = {10.1016/j.aml.2016.12.019},
author = {Both, J. W. and Borregales, M. and Nordbotten, J. M. and Kumar, K. and Radu, F. A. },
}

@article{Brezina2024,
author = {Březina, J. and Stebel, J.},
title = {Discrete fracture-matrix model of poroelasticity},
journal = {ZAMM - Journal of Applied Mathematics and Mechanics / Zeitschrift für Angewandte Mathematik und Mechanik},
volume = {104},
number = {4},
pages = {e202200469},
doi = {10.1002/zamm.202200469},
year = {2024}
}

@article{BYZ2017,
	author = {{Bukač, M.} and {Yotov, I.} and {Zunino, P.}},
	title = {Dimensional model reduction for flow through fractures   in poroelastic media},
	DOI= "10.1051/m2an/2016069",
	journal = {ESAIM: Mathematical Modelling and Numerical Analysis},
	year = 2017,
	volume = 51,
	number = 4,
	pages = "1429-1471",
}

@article{Garipov2016,
  author    = {Garipov, T. T. and Karimi-Fard, M. and Tchelepi, H. A.},
  title     = {Discrete fracture model for coupled flow and geomechanics},
  journal   = {Computational Geosciences},
  year      = {2016},
  volume    = {20},
  number    = {1},
  pages     = {149--160},
  doi       = {10.1007/s10596-015-9554-z},
  issn      = {1573-1499}
}

@article{Kadeethum2020,
author = {Kadeethum, T. and Lee, S. and Nick, H. M.},
title = {{Finite Element Solvers for Biot’s Poroelasticity Equations in Porous Media}},
journal = {Mathematical Geosciences},
volume = {52},
pages = {977–1015},
doi = {10.1007/s11004-020-09893-y},
year = {2020}
}

@article{KIM2011CMAME,
title = {Stability and convergence of sequential methods for coupled flow and geomechanics: Fixed-stress and fixed-strain splits},
journal = {Computer Methods in Applied Mechanics and Engineering},
volume = {200},
number = {13},
pages = {1591-1606},
year = {2011},
issn = {0045-7825},
doi = {10.1016/j.cma.2010.12.022},
author = {J. Kim and H.A. Tchelepi and R. Juanes}
}

@article{KIM2011Undrained,
title = {Stability and convergence of sequential methods for coupled flow and geomechanics: Drained and undrained splits},
journal = {Computer Methods in Applied Mechanics and Engineering},
volume = {200},
number = {23},
pages = {2094-2116},
year = {2011},
issn = {0045-7825},
doi = {https://doi.org/10.1016/j.cma.2011.02.011},
url = {https://www.sciencedirect.com/science/article/pii/S0045782511000466},
author = {J. Kim and H.A. Tchelepi and R. Juanes}
}

@article{JHA2014,
author = {Jha, B. and Juanes, R},
title = {Coupled multiphase flow and poromechanics: A computational model of pore pressure effects on fault slip and earthquake triggering},
journal = {Water Resources Research},
volume = {50},
number = {5},
pages = {3776-3808},
doi = {10.1002/2013WR015175},
year = {2014}
}

@article{Mikelic2013,
  author    = {Mikeli{\'c}, A. and Wheeler, M. F.},
  title     = {Convergence of iterative coupling for coupled flow and geomechanics},
  journal   = {Computational Geosciences},
  year      = {2013},
  volume    = {17},
  number    = {3},
  pages     = {455--461},
  doi       = {10.1007/s10596-012-9318-y},
  issn      = {1573-1499}
}

@article{S2024Finel,
author = {Scialò, S.},
publisher = {Elsevier B.V},
title = {A five field formulation for flow simulations in porous media with fractures and barriers via an optimization based domain decomposition method},
volume = {238},
year = {2024},
copyright = {2024 The Author(s)},
issn = {0168-874X},
journal = {Finite elements in analysis and design},
language = {eng},
pages = {104204},
doi={10.1016/j.finel.2024.104204}
}

@article{SETTARI1998_FixedStress,
    author = {Settari, A. and Mounts, F. M.},
    title = {A Coupled Reservoir and Geomechanical Simulation System},
    journal = {SPE Journal},
    volume = {3},
    number = {03},
    pages = {219-226},
    year = {1998},
    month = {09},
    issn = {1086-055X},
    doi = {10.2118/50939-PA}
}

@article{Storvik2019,
author = {Storvik, E. and Both, J. W. and Kumar, K. and Nordbotten, J. M. and Radu, F. A.},
title = {On the optimization of the fixed-stress splitting for Biot's equations},
journal = {International Journal for Numerical Methods in Engineering},
volume = {120},
number = {2},
pages = {179-194},
doi = {10.1002/nme.6130},
year = {2019}
}

@article{SukumarXFEM2000,
author = {Sukumar, N. and Moës, N. and Moran, B. and Belytschko, T.},
title = {Extended finite element method for three-dimensional crack modelling},
journal = {International Journal for Numerical Methods in Engineering},
volume = {48},
number = {11},
pages = {1549-1570},
doi = {10.1002/1097-0207(20000820)48:11<1549::AID-NME955>3.0.CO;2-A},
year = {2000}
}

@book{terzaghi,
author = {Terzaghi,K.},
address = {London New York},
title = {Theoretical soil mechanics},
edition = {4th ed.},
language = {eng},
publisher = {Chapman and Hall Wiley},
year = {1947},
}

@article{Tetgen2015,
author = {Si, H.},
title = {TetGen, a Delaunay-Based Quality Tetrahedral Mesh Generator},
year = {2015},
publisher = {Association for Computing Machinery},
address = {New York, NY, USA},
volume = {41},
number = {2},
issn = {0098-3500},
doi = {10.1145/2629697},
journal = {ACM Transactions on Mathematical Software},
}

@article{TURSKA1993,
title = {On convergence conditions of partitioned solution procedures for consolidation problems},
journal = {Computer Methods in Applied Mechanics and Engineering},
volume = {106},
number = {1},
pages = {51-63},
year = {1993},
issn = {0045-7825},
doi = {https://doi.org/10.1016/0045-7825(93)90184-Y},
url = {https://www.sciencedirect.com/science/article/pii/004578259390184Y},
author = {E. Turska and B.A. Schrefler},
}

@article{Xfem99_1,
author = {Belytschko, T. and Black, T.},
title = {Elastic crack growth in finite elements with minimal remeshing},
journal = {International Journal for Numerical Methods in Engineering},
volume = {45},
number = {5},
pages = {601-620},
doi = {10.1002/(SICI)1097-0207(19990620)45:5<601::AID-NME598>3.0.CO;2-S},
year = {1999}
}

@article{xfem99_2,
author = {Moës, Nicolas and Dolbow, John and Belytschko, Ted},
title = {A finite element method for crack growth without remeshing},
journal = {International Journal for Numerical Methods in Engineering},
volume = {46},
number = {1},
pages = {131-150},
doi = {10.1002/(SICI)1097-0207(19990910)46:1<131::AID-NME726>3.0.CO;2-J},
year = {1999}
}

@article{Aagaard2013,
author = {Aagaard, B. T. and Knepley, M. G. and Williams, C. A.},
title = {A domain decomposition approach to implementing fault slip in finite-element models of quasi-static and dynamic crustal deformation},
journal = {Journal of Geophysical Research: Solid Earth},
volume = {118},
number = {6},
pages = {3059-3079},
doi = {https://doi.org/10.1002/jgrb.50217},
year = {2013}
}

@article{XfemReview,
author = {Fries, Thomas-Peter and Belytschko, Ted},
title = {The extended/generalized finite element method: An overview of the method and its applications},
journal = {International Journal for Numerical Methods in Engineering},
volume = {84},
number = {3},
pages = {253-304},
doi = {10.1002/nme.2914},
year = {2010}
}

@article{MJR2005,
author = {V. Martin and J. Jaffr\'{e} and J.E. Roberts},
title = {Modeling Fractures and Barriers as Interfaces for Flow in Porous Media},
journal = {SIAM Journal on Scientific Computing},
volume = {26},
number = {5},
pages = {1667-1691},
year = {2005},
doi = {10.1137/S1064827503429363}
}

@article{Burman2014,
author = {Burman, Erik},
title = {Projection stabilization of Lagrange multipliers for the imposition of constraints on interfaces and boundaries},
journal = {Numerical Methods for Partial Differential Equations},
volume = {30},
number = {2},
pages = {567-592},
doi = {https://doi.org/10.1002/num.21829},
year = {2014}
}
\end{document}